\documentclass[10pt, leqno]{amsart}

\usepackage{enumerate}
\usepackage[dvips]{graphics}
\usepackage{accents,amsmath,amssymb,hyperref,mathrsfs,mathtools}

\newcommand{\doublespace}
{\addtolength{\baselineskip}{0.15\baselineskip}}

\theoremstyle{plain}
\newtheorem{thm}{Theorem}[section]
\newtheorem{lem}[thm]{Lemma}

\newtheorem{prop}[thm]{Proposition}

\theoremstyle{definition}
\newtheorem{defn}[thm]{Definition}
\newtheorem{notn}[thm]{Notation}
\newtheorem{exmp}[thm]{Example}
\newtheorem{rmk}[thm]{Remark}

\newcounter{equationnumber}
\renewcommand{\theequation}{\thesection.\arabic{equation}}
\def\mathletters{
\addtocounter{equation}{1}
\edef\@currentlabel{\theequation}
\setcounter{equationnumber}{\value{equation}}
\setcounter{equation}{0}
\edef\theequation{\@currentlabel\noexpand\alph{equation}}
}

\title{An analogue of the L\'{e}vy-Hin\v{c}in formula for bi-free infinitely divisible distributions}
\author{Yinzheng Gu$^\ddagger$, Hao-Wei Huang$^\ddagger$, and James A. Mingo$^\ddagger$}
\address{Department of Mathematics and Statistics, Queen's University, Jeffery Hall, Kingston, Ontario, K7L 3N6, Canada}
\email{gu.y@queensu.ca, hwhuang@mast.queensu.ca, mingo@mast.queensu.ca}
\thanks{$^\ddagger$Research supported by a Discovery Grant from the Natural Sciences and Engineering Research Council of Canada}

\begin{document}

\maketitle \doublespace \pagestyle{myheadings} \thispagestyle{plain}
\markboth{ }{ }

\begin{abstract}

In this paper, we derive the bi-free analogue of the L\'{e}vy-Hin\v{c}in formula for compactly supported planar probability measures which are infinitely divisible with respect to the additive bi-free convolution introduced by Voiculescu. We also provide examples of bi-free infinitely divisible distributions with their bi-free L\'{e}vy-Hin\v{c}in representations. Furthermore, we construct the bi-free L\'{e}vy processes and the additive bi-free convolution semigroups generated by compactly supported planar probability measures.

\end{abstract}


\section{Introduction}

Around thirty years ago, Voiculescu introduced free probability theory in order to attack some problems in the theory of operator algebras. He introduced free independence, an analogue of the classical notion of independence, with the intention of studying these problems in a probabilistic framework. The (additive) free convolution $\boxplus$, an analogue of the classical convolution $*$, is a binary operation on the set of compactly supported probability measures on $\mathbb{R}$ which corresponds to the sum of free random variables in a non-commutative probability space. This operation was later generalized to the set $\mathcal{M}_\mathbb{R}$ of Borel probability measures on $\mathbb{R}$ by Bercovici and Voiculescu \cite{BV93}. One of the essential functions in the theory is the free $\mathcal{R}$-transform of measures in $\mathcal{M}_\mathbb{R}$, which linearizes the additive free convolution \cite{V86, BV93}. The combinatorial apparatus of free cumulants and the lattice of non-crossing partitions, introduced by Speicher \cite{S97}, also play important roles in free probability theory for the study of sums and products of free $n$-tuples of random variables.

Either in classical or free probability theory, infinitely divisible probability measures play a central role. A probability measure $\mu \in \mathcal{M}_\mathbb{R}$ is said to be $*$-infinitely divisible (resp. $\boxplus$-infinitely divisible) if, for every $n\in\mathbb{N}$, it can be represented as an $n$-fold classical (resp. free) convolution of some probability measure in $\mathcal{M}_\mathbb{R}$. Measures which are $*$-infinitely divisible were first studied by de Finetti, Kolmogorov, L\'{e}vy, and Hin\v{c}in as they arise as the limit distributions of sums of independent random variables within a triangular array. The logarithm of the Fourier transform of a $*$-infinitely divisible distribution permits an integral representation called the L\'{e}vy-Hin\v{c}in representation. On the other hand, $\boxplus$-infinitely divisible distributions were first investigated by Voiculescu, and since then the theory has been well-developed. The theory of $\boxplus$-infinitely divisible distributions generalize the free central limit theorem as they appear in the limits of triangular arrays of freely independent random variables. Likewise, the free L\'{e}vy-Hin\v{c}in formula gives a complete description of $\boxplus$-infinitely divisible distributions. We refer the reader to \cite{BV92, NS96, BP99} for more details.

Furthermore, infinitely divisible distributions are closely related to stationary processes with independent increments. From the theoretical and applied points of view, L\'{e}vy processes form a very important research area in classical probability theory. Such processes in free probability theory also receive a lot of attention. By analogy with classical probability theory, the distributions $\mu_t$ of $X_t$ in an (additive) free L\'{e}vy process $(X_t)_{t \geq 0}$ satisfy the properties that $\mu_0 = \delta_0$ (the point mass at $0$), the weak convergence of $\mu_t$ to $\delta_0$ as $t \rightarrow 0^+$, and the semigroup property relative to the free convolution:
\begin{equation}\label{semigroup1}
\mu_s \boxplus \mu_t = \mu_{s + t},\;\;\;\;\;s,t \geq 0.
\end{equation}
As in the classical case, the distribution $\mu_1$ of $X_1$ is $\boxplus$-infinitely divisible.

For $n \in \mathbb{N}$ and $\mu\in\mathcal{M}_\mathbb{R}$, denote by $\mu_n$ the $n$-fold free convolution of $\mu$. One peculiarity of the free convolution is that the discrete free convolution semigroup $(\mu_n)_{n \in \mathbb{N}}$ can be embedded in a continuous family $(\mu_t)_{t \geq 1}$ which satisfies the semigroup property \eqref{semigroup1} for $s,t \geq 1$. This elegant result for a compactly supported measure $\mu$, proved by Nica and Speicher \cite{NS96}, has no parallel in classical probability theory. The exhibition of explicit random variables whose distributions are measures in $(\mu_t)_{t \geq 1}$ has several applications in random matrix theory.

Recently, Voiculescu \cite{V14, V13} introduced bi-free probability theory in order to study algebras of left operators and algebras of right operators simultaneously. This gives rise to the notions of bi-free cumulants, bi-free $\mathcal{R}$-transform, and the operation of (additive) bi-free convolution. In this paper, we prove a bi-free limit theorem for sums of bi-free pairs of random variables within a triangular array, and define the $\boxplus\boxplus$-infinite divisibility of planar probability measures in a natural manner. One of the main goals of this paper is to characterize the bi-free $\mathcal{R}$-transforms of compactly supported planar probability measures which are $\boxplus\boxplus$-infinitely divisible, and derive their bi-free L\'{e}vy-Hin\v{c}in representations. With the help of a bi-free limit theorem, we are able to provide some examples, such as the bi-free Gaussian and bi-free (compound) Poisson distributions, which are $\boxplus\boxplus$-infinitely divisible. A natural object in the study of the $\boxplus\boxplus$-infinite divisibility of distributions is the extension of the notions of free L\'{e}vy process and free convolution semigroups to the bi-free setting. Another goal of this paper is to prove the existence of the additive bi-free convolution semigroups generated by compactly supported planar probability measures. The main tools used throughout the paper are the combinatorics developed in free and bi-free probability theories.

This paper has four more sections after this introduction. Section $2$ provides background information on the development of bi-free probability theory based on \cite{V14, V13, MN15, CNS15}. Section $3$ contains the derivation of the bi-free L\'{e}vy-Hin\v{c}in formula. Sections $4$ and $5$ are dedicated to the study of bi-free L\'{e}vy processes and additive bi-free convolution semigroups.

\section{Preliminaries}

\subsection{Bi-free independence and bi-free cumulants}

We start by reviewing some definitions and results from \cite{V14, MN15, CNS15}. An ordered pair $(\mathcal{B}, \mathcal{C})$ is said to be a pair of (included) faces in a non-commutative probability space $(\mathcal{A}, \varphi)$ if $\mathcal{B}$ and $\mathcal{C}$ are unital subalgebras of $\mathcal{A}$, in which $\mathcal{B}$ and $\mathcal{C}$ are called the left and right face, respectively. In \cite[Definition 2.6]{V14}, Voiculescu defined bi-free independence for pairs of faces as follows.

\begin{defn}
A family $\pi = \{(\mathcal{B}_k, \mathcal{C}_k)\}_{k \in K}$ of pairs of faces in $(\mathcal{A}, \varphi)$ is said to be \textit{bi-free} if there exists a family of vector spaces with specified vector states $\{(\mathcal{X}_k, \mathring{\mathcal{X}}_k, \xi_k)\}_{k \in K}$ and unital homomorphisms $\ell_k: \mathcal{B}_k \rightarrow \mathcal{L}(\mathcal{X}_k),\;r_k: \mathcal{C}_k \rightarrow \mathcal{L}(\mathcal{X}_k)$ such that the joint distribution of $\pi$ with respect to $\varphi$ is equal to the joint distribution of $\tilde{\pi} = \{(\lambda_k \circ \ell_k(\mathcal{B}_k), \rho_k \circ r_k(\mathcal{C}_k))\}_{k \in K}$ with respect to the vacuum state on $\mathcal{L}(\mathcal{X})$, where $(\mathcal{X}, \mathring{\mathcal{X}}, \xi) = *_{k \in K}(\mathcal{X}_k, \mathring{\mathcal{X}}_k, \xi_k)$, $\lambda_k$ and $\rho_k$ are the left and right representations of $\mathcal{L}(\mathcal{X}_k)$ on $\mathcal{L}(\mathcal{X})$.
\end{defn}

Let $I$ and $J$ be index sets. If $(b_i')_{i\in I},\;(b_i'')_{i\in I},\;(c_j')_{j\in J},\;(c_j'')_{j\in J}$ are elements of $\mathcal{A}$, then the two-faced families of non-commutative random variables $(b', c') = ((b_i')_{i \in I}, (c_j')_{j \in J})$ and $(b'', c'') = ((b_i'')_{i \in I}, (c_j'')_{j \in J})$ are said to be bi-free if the associated pairs of faces $(\mathbb{C}\langle b_i' : i \in I\rangle, \mathbb{C}\langle c_j' : j \in J\rangle)$ and $(\mathbb{C}\langle b_i'' : i \in I\rangle, \mathbb{C}\langle c_j'' : j \in J\rangle)$ are bi-free (see \cite[Section 2]{V14}). Moreover, if $(b', c')$ and $(b'', c'')$ are bi-free with joint distributions $\mu'$ and $\mu''$, respectively, then the joint distribution of $((b_i' + b_i'')_{i \in I}, (c_j' + c_j'')_{j \in J})$ is called the additive bi-free convolution of $\mu'$ and $\mu''$, and is denoted by $\mu' \boxplus\boxplus \mu''$ (see \cite[Section 4]{V14}).

It was shown in \cite[Section 5]{V14} that there exist universal polynomials, called bi-free cumulants, on the mixed moments of bi-free pairs of two-faced families of non-commutative random variables which linearize the additive bi-free convolution. However, there were no explicit formulas for the bi-free cumulants. Later, Mastnak and Nica defined $(\ell, r)$-cumulants and combinatorial-bi-free independence in \cite{MN15} as follows.

\begin{defn}
Let $(\mathcal{A}, \varphi)$ be a non-commutative probability space and let $[n]$ denote $\{1, \dots, n\}$ for $n \geq 1$. There exists a family of multilinear functionals
$$(\kappa_\chi: \mathcal{A}^n \rightarrow \mathbb{C})_{n \geq 1, \chi: [n] \rightarrow \{\ell, r\}}$$
which is uniquely determined by the requirement that
$$\varphi(a_1\cdots a_n) = \sum_{\pi \in \mathcal{P}^{(\chi)}(n)}\left(\prod_{V \in \pi}\kappa_{\chi | V}((a_1, \dots, a_n) | V)\right)$$
for every $n \geq 1$, $\chi: [n] \rightarrow \{\ell, r\}$, and $a_1, \dots, a_n \in \mathcal{A}$. These $(\kappa_\chi)_{n \geq 1, \chi: [n] \rightarrow \{\ell, r\}}$ are called the \textit{$(\ell, r)$-cumulants} of $(\mathcal{A}, \varphi)$.
\end{defn}

Given $\chi: [n] \rightarrow \{\ell, r\}$, $n \geq 1$, such that $\chi^{-1}(\{\ell\}) = \{i_1 < \cdots < i_p\}$ and $\chi^{-1}(\{r\}) = \{j_1 < \cdots < j_{n - p}\}$, the set of partitions $\mathcal{P}^{(\chi)}(n)$ appearing in the above definition is obtained by applying the permutation $\sigma_\chi \in S_n$ to the elements of $\mathrm{NC}(n)$, the set of non-crossing partitions of $[n]$, where $\sigma_\chi$ is defined by
$$\sigma_\chi(k) = \begin{cases}
i_k, &\text{if}\;k \leq p,\\
j_{n - k + 1}, &\text{if}\;k > p.
\end{cases}$$

\begin{defn}
Let $(\mathcal{A}, \varphi)$ be a non-commutative probability space and let $a_1, \dots, a_d$, $b_1, \dots, b_d$ be elements of $\mathcal{A}$. Denoting $c_{i ; \ell} = a_i$ and $c_{i ; r} = b_i$ for $1 \leq i \leq d$, the two-faced pairs $(a_1, b_1), \dots, (a_d, b_d)$ are said to be \textit{combinatorially-bi-free} if
$$\kappa_\chi(c_{i_1 ; \chi(i_1)}, \dots, c_{i_n ; \chi(i_n)}) = 0$$
whenever $n \geq 2$, $\chi: [n] \rightarrow \{\ell, r\}$, $i_1, \dots, i_n \in [d]$, and there exist $1 \leq p < q \leq n$ such that $i_p \neq i_q$.
\end{defn}

After giving the above definition, Mastnak and Nica asked the question of whether combinatorial-bi-free independence was equivalent to bi-free independence, and it was answered affirmatively by Charlesworth, Nelson, and Skoufranis in \cite{CNS15} using bi-non-crossing partitions. We refer the reader to \cite[Section 2]{CNS15} for details. For $\chi: [n] \rightarrow \{\ell, r\}$, $n \geq 1$, such that $\chi^{-1}(\{\ell\}) = \{i_1 < \cdots < i_p\}$ and $\chi^{-1}(\{r\}) = \{j_1 < \cdots < j_{n - p}\}$, the set of bi-non-crossing partitions $\mathrm{BNC}(\chi)$ defined in \cite[Section 2]{CNS15} coincides with $\mathcal{P}^{(\chi)}(n)$ which, from another diagrammatic point of view, consists of the non-crossing partitions of $[n]$ such that the numbers $1, \dots, n$ are rearranged according to the total order $<_\chi$ on $[n]$
defined by
$$i_1 <_\chi \cdots <_\chi i_p <_\chi j_{n - p} <_\chi \cdots <_\chi j_1.$$
For this reason, we also denote $\mathcal{P}^{(\chi)}(n) = \mathrm{BNC}(\chi)$ by $\mathrm{NC}_\chi(n)$. As lattices with respect to reverse refinement order, $\mathrm{NC}_\chi(n)$ is isomorphic to $\mathrm{NC}(n)$, thus the M\"{o}bius function $\mu_\chi$ on $\mathrm{NC}_\chi(n)$ is given by
$$\mu_\chi(\tau, \pi) = \mu(\sigma_\chi^{-1}\cdot\tau, \sigma_\chi^{-1}\cdot\pi),$$
where $\mu$ denotes the M\"{o}bius function on $\mathrm{NC}(n)$. Finally, as shown in \cite[Sections 3, 4]{CNS15}, the $(\ell, r)$-cumulants are the same as the bi-free cumulants, and we have the moment-cumulant formulas
$$\varphi(a_1\cdots a_n) = \sum_{\pi \in \mathrm{NC}_\chi(n)}\kappa_\pi^\chi(a_1, \dots, a_n)$$
and
$$\kappa_n^\chi(a_1, \dots, a_n) = \sum_{\pi \in \mathrm{NC}_\chi(n)}\varphi_\pi(a_1, \dots, a_n)\mu_\chi(\pi, 1_n)$$
for all $a_1, \dots, a_n \in \mathcal{A}$, where $\kappa_n^\chi(a_1, \dots, a_n) = \kappa_\chi(a_1, \dots, a_n)$ and
$\kappa_\pi^\chi(a_1, \dots, a_n)$ factors according to the blocks of $\pi$ by the multiplicativity of the family $(\kappa_n^\chi)_{n \geq 1, \chi: [n] \rightarrow \{\ell, r\}}$.

\subsection{Free and bi-free $\mathcal{R}$-transforms}

Recall that the joint distribution of a family $(a_i)_{i \in I}$ of random variables in a non-commutative probability space $(\mathcal{A}, \varphi)$ is the linear functional $\mu$ on the algebra $\mathbb{C}\langle X_i: i \in I\rangle$ of non-commutative polynomials in $|I|$ variables satisfying $\mu(P) = \varphi(h(P))$ for all $P \in \mathbb{C}\langle X_i: i \in I\rangle$, where $h: \mathbb{C}\langle X_i: i\in I\rangle \rightarrow \mathcal{A}$ is the unital algebra homomorphism such that $h(X_i) = a_i$.

If $a$ is a self-adjoint random variable in a $C^*$-probability space $(\mathcal{A},\varphi)$, then its distribution $\mu_a$ belongs to $\mathcal{M}_\mathbb{R}$. The Cauchy transform (or one-variable Green's function) of $a$ is defined as
$$G_a(z) = \varphi((z - a)^{-1}),$$
and the free $\mathcal{R}$-transform of $a$ is defined as
$$\mathcal{R}_a(z) = \sum_{n \geq 0}\kappa_{n + 1}(\underbrace{a, \dots, a}_{n + 1\;\text{times}})z^n,$$
where $(\kappa_n)_{n \in \mathbb{N}}$ are the free cumulants of $(\mathcal{A},\varphi)$. It turns out that the functions $G_a$ and $\mathcal{R}_a$ are analytic in a neighbourhood of $\infty$ and $0$, respectively, and the function
$$K_a(z) = \mathcal{R}_a(z) + \frac{1}{z}$$ satisfies $G_a(K_a(z)) = z$. One of the most important properties of the free $\mathcal{R}$-transform is that it linearizes the additive free convolution in the sense that
$$\mathcal{R}_{a' + a''}(z) = \mathcal{R}_{a'}(z) + \mathcal{R}_{a''}(z)$$
if $a'$ and $a''$ are free self-adjoint random variables in $(\mathcal{A},\varphi)$ or, equivalently,
$$\mathcal{R}_{\mu_{a'} \boxplus \mu_{a''}}(z) = \mathcal{R}_{\mu_{a'}}(z) + \mathcal{R}_{\mu_{a''}}(z),$$
which holds in a neighbourhood of $0$ \cite{VDN92}.

Analogously, if $(a, b)$ is a two-faced pair of commuting self-adjoint random variables in a $C^*$-probability space $(\mathcal{A},\varphi)$, i.e. $a = a^*$, $b = b^*$, and $[a, b] = 0$, then the distribution $\mu_{(a, b)}$ of $(a, b)$ is a Borel probability measure on $\mathbb{R}^2$ and the two-dimensional Cauchy transform (or two-variable Green's function) of $(a, b)$ is defined as
$$G_{(a, b)}(z, w) = \varphi((z - a)^{-1}(w - b)^{-1}),$$
which is an analytic function in a neighbourhood of
$\infty \times \infty$.

Note that if $(a, b)$ is a general two-faced pair in a $C^*$-probability space $(\mathcal{A},\varphi)$, then the bi-free cumulant $\kappa_n^\chi$ of $(a, b)$ depends on $\chi: [n] \rightarrow \{\ell, r\}$. Since we are interested in the case where $a$ and $b$ are commuting self-adjoint random variables, it turns out that all the bi-free cumulants of $(a, b)$ are real, and $\kappa_n^\chi$ depends on $\chi$ only through $|\chi^{-1}(\{\ell\})|$ and $|\chi^{-1}(\{r\})|$. Moreover, the commutativity of $a$ and $b$ implies that every bi-free cumulant of $(a, b)$ is a special free cumulant.

\begin{lem}\label{FreeBifreeK}
Let $(a, b)$ be a two-faced pair in a $C^*$-probability space $(\mathcal{A},\varphi)$ such that $a = a^*$, $b = b^*$, and $[a,b]=0$. Denote the free and bi-free cumulants of $(a, b)$ by
$$\kappa_{m, n}(a, b) = \kappa_{m + n}(\underbrace{a, \dots, a}_{m\;\text{times}}, \underbrace{b, \dots, b}_{n\;\text{times}})$$
and
$$\kappa_N^\chi(a, b) = \kappa_N^\chi(c_{\chi(1)}, \dots, c_{\chi(N)}),$$
respectively, where $\chi: [N] \rightarrow \{\ell, r\}$, $c_\ell = a$ and $c_r = b$ for $1 \leq k \leq N$. Then $\kappa_{m, n}(a, b) = \kappa_{m + n}^\chi(a, b)$ for all $\chi: [m + n] \rightarrow \{\ell, r\}$ such that $|\chi^{-1}(\{\ell\})| = m$ and $|\chi^{-1}(\{r\})| = n$.
\end{lem}

\begin{proof}
By the moment-cumulant formulas, we have
$$\kappa_{m, n}(a, b) = \sum_{\pi \in \mathrm{NC}(m + n)} \varphi_\pi(\underbrace{a, \dots, a}_{m\;\text{times}}, \underbrace{b, \dots, b}_{n\;\text{times}})\mu(\pi, 1_{m + n})$$
and
$$\kappa_{m + n}^\chi(a, b) = \sum_{\pi \in \mathrm{NC}_\chi(m + n)}\varphi_\pi(c_{\chi(1)}, \dots, c_{\chi(m + n)})\mu_\chi(\pi, 1_{m + n}),$$
where $\mu$ and $\mu_\chi$ denote the M\"{o}bius functions on $\mathrm{NC}(m + n)$ and $\mathrm{NC}_\chi(m + n)$, respectively. For each partition $\pi \in \mathrm{NC}(m + n)$, which has a linear non-crossing diagram associated to it, the linear diagram of the corresponding partition $\tilde{\pi} = \sigma_\chi\cdot\pi \in \mathrm{NC}_\chi(m + n)$ under the bijection $\sigma_\chi : \mathrm{NC}(m + n) \rightarrow \mathrm{NC}_\chi(m + n)$ is obtained by relabelling the numbers $1, \dots, m + n$ in the linear diagram of $\pi$ with $i_1, \dots, i_m, j_n, \dots, j_1$ where $\{i_1 <
\cdots < i_m\} = \chi^{-1}(\{\ell\})$ and $\{j_1 < \cdots < j_n\} = \chi^{-1}(\{r\})$. Since $a$ and $b$ commute, we have
$$\varphi_\pi(\underbrace{a, \dots, a}_{m\;\text{times}}, \underbrace{b, \dots, b}_{n\;\text{times}}) = \varphi_{\tilde{\pi}}(c_{\chi(1)}, \dots, c_{\chi(m + n)})$$
for every $\pi \in \mathrm{NC}(m + n)$. Moreover, since
$$\mu_\chi(\tilde{\pi}, 1_{m + n}) = \mu(\sigma_\chi^{-1}\cdot\tilde{\pi}, 1_{m + n}) = \mu(\pi, 1_{m + n})$$
for every $\tilde{\pi} \in \mathrm{NC}_\chi(m + n)$, the assertion follows.
\end{proof}

\begin{notn}
Let $(a, b)$ be as above and $m, n \geq 0$ such that $m + n \geq 1$. We extend the notations used in the above lemma for the free and bi-free cumulants of $(a, b)$ to all of $\mathrm{NC}(m + n)$ and $\mathrm{NC}_\chi(m + n)$, where $\chi: [m + n] \rightarrow \{\ell, r\}$ such that $|\chi^{-1}(\{\ell\})| = m$ and $|\chi^{-1}(r)| = n$. That is, for $\pi$ in $\mathrm{NC}(m + n)$ or $\mathrm{NC}_\chi(m + n)$, we put
$$\kappa_\pi(a, b) = \kappa_\pi(\underbrace{a, \dots, a}_{m\;\text{times}},\underbrace{b, \dots, b}_{n\;\text{times}})$$
and
$$\kappa_\pi^\chi(a, b) = \kappa_\pi^\chi(c_{\chi(1)}, \dots, c_{\chi(m + n)}),$$
where $c_\ell = a$ and $c_r = b$ for $1 \leq k \leq m + n$. Similarly, we put
$$\varphi_\pi(a, b) = \varphi_\pi(\underbrace{a, \dots, a}_{m\;\text{times}},\underbrace{b, \dots, b}_{n\;\text{times}})$$
for $\pi$ in $\mathrm{NC}(m + n)$.
\end{notn}

For a two-faced pair $(a, b)$ in a $C^*$-probability space, let
$$\mathcal{R}_{(a, b)}(z, w) = \sum_{\substack{m, n \geq 0\\m + n \geq 1}}\kappa_{m, n}(a, b)z^mw^n,$$
be the (partial) bi-free $\mathcal{R}$-transform of $(a, b)$. Then we have the following relation for bi-free $\mathcal{R}$-transforms, which was proved in \cite{V13} using analytic techniques and re-derived in \cite{S14} via combinatorics.

\begin{thm}\label{bifreeR}
The following equality of germs of holomorphic functions holds in a neighbourhood of $(0, 0)$ in $\mathbb{C}^2$:
$$\mathcal{R}_{(a, b)}(z, w) = 1 + z\mathcal{R}_a(z) + w\mathcal{R}_b(w) - \dfrac{zw}{G_{(a, b)}(K_a(z), K_b(w))}.$$
\end{thm}

The bi-free $\mathcal{R}$-transform is an analogue of the free $\mathcal{R}$-transform in the bi-free setting. More precisely, if $(a', b')$ and $(a'', b'')$ are bi-free two-faced pairs, then we have
$$\mathcal{R}_{(a' + a'', b' + b'')}(z, w)=\mathcal{R}_{(a', b')}(z, w) + \mathcal{R}_{(a'', b'')}(z, w)$$
for $(z, w)$ near $(0, 0)$.

\subsection{Moment sequences}

Let $\mathbb{Z}_+^2 = \{(m, n) : m, n \in \mathbb{N} \cup \{0\}\}$. Given a $2$-sequence $R = \{R_{m, n}\}_{(m, n) \in \mathbb{Z}_+^2}$ with $R_{0, 0} > 0$, one can equip the algebra $\mathbb{C}[s, t]$ of polynomials in commuting variables $s$ and $t$ with a sesquilinear form $[\cdot, \cdot]_R$ satisfying
\begin{equation}\label{sesquilinear}
[s^{m_1}t^{n_1}, s^{m_2}t^{n_2}]_R = R_{m_1 + m_2, n_1 + n_2}
\end{equation}
for $(m_1, n_1), (m_2, n_2) \in \mathbb{Z}_+^2$. Note that if $p = \sum_{j = 1}^\ell c_js^{m_j}t^{n_j} \in \mathbb{C}[s, t]$, then
$$[p, p]_R = \sum_{j, k = 1}^\ell c_j\overline{c}_kR_{m_j + m_k, n_j + n_k}.$$

Recall that the $2$-sequence $\{\kappa_{m, n}\}_{(m, n) \in \mathbb{Z}_+^2 \backslash \{(0, 0)\}}$ of the free cumulants of a pair of commuting self-adjoint random variables $(a, b)$ in some $C^*$-probability space contain the full information about $(a, b)$. It turns out that the study of such $2$-sequences is closely related to the two-parameter moment problems \cite{D57, PV99}. The following result is from \cite{D57}.

\begin{thm}\label{moment1}
A $2$-sequence $R = \{R_{m, n}\}_{(m, n) \in \mathbb{Z}_+^2}$ with $R_{0, 0} > 0$ is a moment $2$-sequence, i.e. there exists a finite positive Borel measure $\rho$ on $\mathbb{R}^2$ such that
$$R_{m, n} = \int_{\mathbb{R}^2}s^mt^n\;d\rho(s, t),\;\;\;\;\;(m, n) \in \mathbb{Z}_+^2,$$
if there exists a finite number $L > 0$ with the following properties: for all $p \in \mathbb{C}[s,t]$,
\begin{enumerate}[$(1)$]
\item $[p, p]_R \geq 0$,

\item $|[sp, p]_R| \leq L\cdot[p, p]_R$ and $|[p, tp]_R| \leq L\cdot[p,p]_R$ hold.
\end{enumerate}
If these conditions hold, then the representing measure $\rho$ of $R$ is compactly supported on $[-L, L]^2$ and uniquely determined.
\end{thm}

\begin{rmk}\label{moment2}
\begin{enumerate}[$(1)$]
\item Given a $2$-sequence $\{R_{m, n}\}_{(m, n) \in \mathbb{Z}_+^2}$ with $R_{0, 0} = 1$, there exists a non-commutative probability space $(\mathcal{A},\varphi)$ and commuting random variables $a, b \in \mathcal{A}$ such that $\varphi(a^mb^n) = R_{m, n}$ for all $(m, n) \in \mathbb{Z}_+^2$.

\item The condition (1) in Theorem \ref{moment1} together with the condition that for every fixed $m_0$ and $n_0$ in $\mathbb{N} \cup \{0\}$ the $1$-sequences $\{R_{m, 2n_0}\}_{m \geq 0}$ and $\{R_{2m_0, n} + R_{0, n}\}_{n \geq 0}$ are determined moment sequences, guarantee that $R$ is a determined moment sequence (see \cite[Section 1.2]{D57}). In this case, however, the representing measure of $R$ may not be compactly supported. In general, one can determine whether an $n$-sequence is a moment sequence by embedding the sequence into a higher dimensional space. We refer the reader to \cite{PV99} for more details.
\end{enumerate}
\end{rmk}

\section{Bi-free infinitely divisible distributions}

In this section, we study the $\boxplus\boxplus$-infinite divisibility of compactly supported probability measures on $\mathbb{R}^2$ and provide the bi-free analogue of the L\'{e}vy-Hin\v{c}in formula.

\subsection{A bi-free limit theorem for bipartite systems}

The non-commutative probability spaces considered throughout the paper are assumed to be bipartite, i.e. all left variables commute with all right variables.

\begin{thm}\label{BFlimthm}
For each $N \in \mathbb{N}$, let $\{(a_{N ; k}, b_{N ; k})\}_{k = 1}^N$ be a row of two-faced pairs in some non-commutative probability space $(\mathcal{A}_N, \varphi_N)$. Furthermore, assume each $\mathcal{A}_N$ is bipartite and the two-faced pairs $(a_{N ; 1}, b_{N ; 1}), \dots ,(a_{N ; N}, b_{N ; N})$ are bi-free and identically distributed. Then the following two statements are equivalent.
\begin{enumerate}[$(1)$]
\item There is a two-faced pair $(a, b)$ in some non-commutative probability space $(\mathcal{A}, \varphi)$ such that $[a, b] = 0$ and
$$\left(\sum_{k = 1}^Na_{N ; k}, \sum_{k = 1}^Nb_{N ;k}\right) \xrightarrow{\textnormal{dist}} (a, b).$$

\item For all $m, n \geq 0$ with $m + n \geq 1$, the limits $\displaystyle\lim_{N \rightarrow \infty}N\cdot\varphi_N(a_{N ; k}^mb_{N ; k}^n)$ exist and are independent of $k$.
\end{enumerate}
Furthermore, if these assertions hold, then the bi-free cumulants of $(a, b)$ are given by
$$\kappa_{m + n}^\chi(a, b) = \lim_{N \rightarrow \infty}N\cdot\varphi_N(a_{N ; k}^mb_{N ; k}^n),$$
where $\chi: [m + n] \rightarrow \{\ell, r\}$ satisfies $|\chi^{-1}(\{\ell\})| = m$ and $|\chi^{-1}(\{r\})| = n$.
\end{thm}

\begin{rmk}
\begin{enumerate}[$(1)$]
\item We follow the usual notion of convergence in distribution in the free probability context. That is, assertion (1) in the theorem above holds if and only if
$$\lim_{N \rightarrow \infty}\varphi_N\left(\left(\sum_{k = 1}^Na_{N ; k}\right)^m\left(\sum_{k = 1}^Nb_{N ; k}\right)^n\right) = \varphi(a^mb^n)$$
for all $(m, n) \in \mathbb{Z}_+^2$ and all mixed moments $\varphi(a^mb^n)$ exist.

\item By Lemma \ref{FreeBifreeK}, the bi-free cumulants $\kappa_{m + n}^\chi(a, b)$ are the same as the free cumulants $\kappa_{m, n}(a, b)$ for all $m, n \geq 0$ such that $m + n \geq 1$.
\end{enumerate}
\end{rmk}

To prove Theorem \ref{BFlimthm}, we need the following lemma which relates convergence of moments to convergence of cumulants.

\begin{lem}\label{MCLemma}
For each $N \in \mathbb{N}$, let $(\mathcal{A}_N, \varphi_N)$ be a non-commutative probability space and let $\kappa^N$ be the corresponding free cumulants. Let $(a_N, b_N)$ be a two-faced pair in $(\mathcal{A}_N, \varphi_N)$ such that $[a_N, b_N] = 0$, then the following two statements are equivalent.
\begin{enumerate}[$(1)$]
\item For all $m, n \geq 0$ with $m + n \geq 1$, the limits $\displaystyle\lim_{N \rightarrow \infty}N\cdot\varphi_N(a_N^mb_N^n)$ exist.

\item For all $m, n \geq 0$ with $m + n \geq 1$, the limits $\displaystyle\lim_{N \rightarrow \infty}N\cdot\kappa_{m, n}^N(a_N, b_N)$ exist.
\end{enumerate}
Furthermore, if these conditions hold, then the limits in \textnormal{(1)} and \textnormal{(2)} are the same.
\end{lem}

\begin{proof}
By the moment-cumulant formulas, we have
$$\lim_{N \rightarrow \infty}N\cdot\varphi_N(a_N^mb_N^n) = \lim_{N \rightarrow \infty}N\cdot\sum_{\pi \in \mathrm{NC}(m + n)}\kappa_\pi^N(a_N, b_N)$$
and
$$\lim_{N \rightarrow \infty}N\cdot\kappa_{m, n}^N(a_N, b_N) = \lim_{N \rightarrow \infty}N\cdot\sum_{\pi \in \mathrm{NC}(m + n)}(\varphi_N)_\pi(a_N, b_N)\mu(\pi, 1_{m + n}).$$
If the first statement (resp. second statement) is true, then the only non-vanishing term on the right-hand side of the
second equation (resp. first equation) above corresponds to $\pi = 1_{m + n}$, and the assertion follows.
\end{proof}

\begin{proof}[Proof of $\mathbf{Theorem\;\ref{BFlimthm}}$]
Suppose first that assertion $(1)$ holds. Since we will not be using the bi-free cumulants of $(\mathcal{A}, \varphi)$ until the end of the proof, we let $\kappa^\chi$ denote the bi-free cumulants of $(\mathcal{A}_N, \varphi_N)$ for now. For $m, n \geq 0$ such that $m + n \geq 1$, we have
\begin{align*}
\varphi(a^mb^n) &= \lim_{N \rightarrow \infty}\varphi_N\left(\left(\sum_{k = 1}^Na_{N ; k}\right)^m\left(\sum_{k = 1}^N
b_{N ; k}\right)^n\right)\\
&= \lim_{N \rightarrow \infty}\sum_{\substack{r(1), \dots ,r(m)\\s(1), \dots, s(n) = 1}}^N\varphi_N(a_{N ; r(1)}\cdots
a_{N ; r(m)}b_{N ; s(1)}\cdots b_{N ; s(n)})\\
&= \lim_{N \rightarrow \infty}\sum_{\substack{r(1), \dots, r(m)\\s(1), \dots, s(n) = 1}}^N\sum_{\tau \in \mathrm{NC}_\chi(m + n)}\kappa_\tau^\chi(a_{N ; r(1)}, \dots, a_{N ; r(m)}, b_{N ; s(1)}, \dots, b_{N ; s(n)})\\
&= \lim_{N \rightarrow \infty}\sum_{\tau \in \mathrm{NC}_\chi(m + n)}N^{|\tau|}\cdot\kappa_\tau^\chi(a_{N ; k}, b_{N ; k}),
\end{align*}
where the last expression, which is independent of $k$, follows from the fact that mixed bi-free cumulants vanish. It remains to show that the limits
$$\lim_{N \rightarrow \infty}N^{|\tau|}\cdot\kappa_\tau^\chi(a_{N ; k}, b_{N ; k})$$
exist for all $\tau \in \mathrm{NC}_\chi(m + n)$, then the special case $\tau = 1_{m + n}$ would give us the existence of
$$\lim_{N \rightarrow \infty}N\cdot\kappa_{m + n}^\chi(a_{N ; k}, b_{N ; k}),$$
which is equal to $\lim_{N \rightarrow\infty}N\cdot\kappa_{m, n}^N(a_{N ; k}, b_{N ; k})$ by Lemma \ref{FreeBifreeK}, and the existence of $\lim_{N \rightarrow \infty}N\cdot\varphi_N(a_{N ; k}^mb_{N ; k}^n)$ would follow from Lemma \ref{MCLemma}. We proceed by induction on $m$ and $n$. If $m = 1$ and $n = 0$, then
$$\lim_{N \rightarrow \infty}N\cdot\kappa_1^N(a_{N ; k}) = \lim_{N \rightarrow \infty}N\cdot\varphi_N(a_{N ; k})
=\lim_{N \rightarrow \infty}\varphi_N(a_{N ; 1} + \cdots + a_{N ; N}) = \varphi(a)$$
exists. The case $m = 0$ and $n = 1$ is similar. If $m = n = 1$, then we have
$$\varphi(ab) = \lim_{N \rightarrow \infty}\big(N^2\cdot\kappa_1^N(a_{N ; k})\kappa_1^N(b_{N ; k}) + N\cdot\kappa_{1, 1}^N(a_{N ; k}, b_{N ; k})\big).$$
Since $\varphi(ab)$, $\lim_{N \rightarrow \infty}N\cdot\kappa_1^N(a_{N ; k})$, and $\lim_{N \rightarrow \infty}N\cdot\kappa_1^N(b_{N ; k})$ all exist, we obtain the existence of $\lim_{N \rightarrow \infty}N\cdot\kappa_{1, 1}^N(a_{N ; k}, b_{N ; k})$. For the inductive step, assume the assertion is true for all $m \leq r$ and
$n \leq s$ such that $r + s \geq 1$. If $m = r + 1$ and $n = s$, then we have
\begin{align*}
\varphi(a^{r + 1}b^s) &= \lim_{N \rightarrow \infty}\sum_{\tau \in \mathrm{NC}_\chi(r + s + 1)}N^{|\tau|}\cdot\kappa_\tau^\chi(a_{N ; k}, b_{N ; k})\\
&= \lim_{N \rightarrow \infty}\left(N\cdot\kappa_{r + 1, s}^N(a_{N ; k}, b_{N ; k}) + L\right),
\end{align*}
where
$$L = \displaystyle\sum_{\substack{\tau \in \mathrm{NC}_\chi(r + s + 1)\\\tau \neq 1_{r + s + 1}}}N^{|\tau|}\cdot\kappa_\tau^\chi(a_{N ; k}, b_{N ; k}).$$
By the induction hypothesis, the limits
$$\lim_{N \rightarrow \infty}N^{|\tau|}\cdot\kappa_\tau^\chi(a_{N ; k}, b_{N ; k})$$
exist for all $\tau \in \mathrm{NC}_\chi(r + s + 1)$ with $\tau \neq 1_{r + s + 1}$, thus $\lim_{N \rightarrow \infty}L$ exists. Since $\varphi(a^{r + 1}b^s)$ also exists by assumption, we obtain the existence of $\lim_{N \rightarrow \infty}N\cdot\kappa_{r + 1, s}^N(a_{N ; k}, b_{N ; k})$ as required. The case $m = r$ and $n = s + 1$ is similar.

Conversely, suppose that assertion $(2)$ holds. For $m, n \geq 0$ such that $m + n \geq 1$, we have
$$\lim_{N \rightarrow \infty}\varphi_N\left(\left(\sum_{k = 1}^Na_{N ; k}\right)^m\left(\sum_{k = 1}^Nb_{N ; k}\right)^n\right) = \lim_{N \rightarrow \infty}\sum_{\tau \in \mathrm{NC}_\chi(m + n)}N^{|\tau|}\cdot\kappa_\tau^\chi(a_{N ; k}, b_{N ; k}).$$
By Lemmas \ref{FreeBifreeK} and \ref{MCLemma}, we have
$$\lim_{N \rightarrow \infty}N^{|\tau|}\cdot\kappa_\tau^\chi(a_{N ; k} , b_{N ; k}) = \lim_{N \rightarrow \infty}N^{|\tau|}\cdot(\varphi_N)_\tau(a_{N ; k}, b_{N ; k})$$
for all $\tau \in \mathrm{NC}_\chi(m + n)$, thus they all exist by assumption. Construct a commuting two-faced pair $(a, b)$ in a non-commutative probability space $(\mathcal{A}, \varphi)$ such that
$$\varphi(a^mb^n) = \lim_{N \rightarrow \infty}\varphi_N\left(\left(\sum_{k = 1}^Na_{N ; k}\right)^m
\left(\sum_{k=1}^Nb_{N;k}\right)^n\right).$$
Such an object always exits. For example, we can realize $a$ and $b$ as $s$ and $t$ in $\mathbb{C}[s, t]$, respectively, and define $\varphi(s^mt^n)$ to be the corresponding limits. Finally, let $\kappa^\chi$ denote the bi-free cumulants of $(\mathcal{A}, \varphi)$, and change the notations for the free and bi-free cumulants of $(\mathcal{A}_N, \varphi_N)$ to $c^N$ and $c^\chi$, respectively. Then we have
\begin{align*}
\varphi(a^mb^n) &= \sum_{\tau \in \mathrm{NC}_\chi(m + n)}\kappa_\tau^\chi(a, b)\\
&= \lim_{N \rightarrow \infty}\sum_{\tau \in \mathrm{NC}_\chi(m + n)}N^{|\tau|}\cdot c_\tau^\chi(a_{N ; k}, b_{N ; k})\\
&= \sum_{\tau \in \mathrm{NC}_\chi(m + n)}\lim_{N \rightarrow
\infty}N^{|\tau|}\cdot c_\tau^\chi(a_{N ; k}, b_{N ; k}).
\end{align*}
By induction on the number of arguments in the bi-free cumulants, we have
$$\kappa_\tau^\chi(a, b) = \lim_{N \rightarrow \infty}N^{|\tau|}\cdot c_\tau^\chi(a_{N ; k}, b_{N ; k})$$
for all $\tau \in \mathrm{NC}_\chi(m + n)$. In particular, when $\tau = 1_{m + n}$, we have
\begin{align*}
\kappa_{m + n}^\chi(a, b) &= \lim_{N \rightarrow \infty}N\cdot c_{m + n}^\chi(a_{N ; k}, b_{N ; k})\\
&= \lim_{N \rightarrow \infty}N\cdot c_{m, n}^N(a_{N ; k}, b_{N ; k})\\
&= \lim_{N \rightarrow \infty}N\cdot\varphi_N(a_{N ; k}^mb_{N ; k}^n)
\end{align*}
by Lemmas \ref{FreeBifreeK} and \ref{MCLemma} again. This completes the proof.
\end{proof}

\subsection{Operators on full Fock spaces}

Recall that the full Fock space over a Hilbert space $\mathcal{H}$ is defined as
$$\mathcal{F}(\mathcal{H}) = \mathbb{C}\Omega \oplus \bigoplus_{n \geq 1}\mathcal{H}^{\otimes n},$$
where $\Omega$ is a distinguished vector of norm one, called the vacuum vector. This gives us a $C^*$-probability space $(B(\mathcal{F}(\mathcal{H})), \tau_\mathcal{H})$, where $\tau_\mathcal{H}$, called the vacuum expectation state, is defined by $\tau_\mathcal{H}(T) = \langle T\Omega, \Omega\rangle$ for $T \in B(\mathcal{F}(\mathcal{H}))$. For our purposes, we are mainly interested in the creation, annihilation, and gauge operators on $\mathcal{F}(\mathcal{H})$, defined as follows.

\begin{defn}
Let $f \in \mathcal{H}$ and $T \in B(\mathcal{H})$.
\begin{enumerate}[$(1)$]
\item The \textit{left creation operator} given by the vector $f$, denoted $\ell(f) \in B(\mathcal{F}(\mathcal{H}))$, is determined by the formulas $\ell(f)(\Omega) = f$ and
$$\ell(f)(\xi_1 \otimes \cdots \otimes \xi_n) = f \otimes \xi_1 \otimes \cdots \otimes \xi_n$$
for all $n \geq 1$ and all $\xi_1, \dots, \xi_n \in \mathcal{H}$. The adjoint $\ell(f)^*$ of $\ell(f)$ is called the \textit{left annihilation operator} given by the vector $f$.

\item The \textit{right creation operator} given by the vector $f$, denoted $r(f) \in B(\mathcal{F}(\mathcal{H}))$, is determined by the formulas $r(f)(\Omega) = f$ and
$$r(f)(\xi_1 \otimes \cdots \otimes \xi_n) = \xi_1 \otimes \cdots \otimes \xi_n \otimes f$$
for all $n \geq 1$ and all $\xi_1, \dots, \xi_n \in \mathcal{H}$. The adjoint $r(f)^*$ of $r(f)$ is called the \textit{right annihilation operator} given by the vector $f$.

\item The \textit{left gauge operator} associated to $T$, denoted $\Lambda_\ell(T) \in B(\mathcal{F}(\mathcal{H}))$, is determined by the formulas $\Lambda_\ell(T)(\Omega) = 0$ and $$\Lambda_\ell(T)(\xi_1 \otimes \cdots \otimes \xi_n) = (T\xi_1) \otimes \xi_2 \otimes \cdots \otimes \xi_n$$
for all $n \geq 1$ and all $\xi_1, \dots, \xi_n \in \mathcal{H}$.

\item The \textit{right gauge operator} associated to $T$, denoted $\Lambda_r(T) \in B(\mathcal{F}(\mathcal{H}))$, is determined by the formulas $\Lambda_r(T)(\Omega) = 0$ and
$$\Lambda_r(T)(\xi_1 \otimes \cdots \otimes \xi_n) = \xi_1 \otimes\cdots \otimes \xi_{n - 1} \otimes (T\xi_n)$$
for all $n \geq 1$ and all $\xi_1, \dots, \xi_n \in \mathcal{H}$.
\end{enumerate}
\end{defn}

\begin{rmk}\label{orthogonal}
Let $\mathcal{H} = \bigoplus_{i \in I}\mathcal{H}_i$, where each $\mathcal{H}_i$ is a Hilbert space. For $i \in I$, let $\mathcal{B}_i$ and $\mathcal{C}_i$ be the $C^*$-algebras generated
by
$$\{\ell(f) : f\in\mathcal{H}_i\} \cup \{\Lambda_\ell(T) : T\mathcal{H}_i \subset \mathcal{H}_i\;\mathrm{and}\;T|_{\mathcal{H} \ominus \mathcal{H}_i} = 0\}$$
and
$$\{r(f) : f \in \mathcal{H}_i\} \cup \{\Lambda_r(T) : T\mathcal{H}_i \subset \mathcal{H}_i\;\mathrm{and}\;T|_{\mathcal{H} \ominus \mathcal{H}_i}=0\},$$
respectively. Then $\{(\mathcal{B}_i, \mathcal{C}_i)\}_{i \in I}$ is a bi-free family in $(B(\mathcal{F(H)}), \tau_\mathcal{H})$ (see \cite[Section 6]{V14}).
\end{rmk}

\begin{prop}\label{commuteprop}
Let $\mathcal{H}$ be a Hilbert space. For any $f, g \in \mathcal{H}$, $T_1 = T_1^*, T_2 = T_2^* \in B(\mathcal{H})$, and $\lambda_1, \lambda_2 \in \mathbb{R}$, the self-adjoint operators
$$a = \ell(f) + \ell(f)^* + \Lambda_\ell(T_1) + \lambda_1\cdot 1\;\;\;\;\;\text{and}\;\;\;\;\;b = r(g) + r(g)^* + \Lambda_r(T_2) + \lambda_2\cdot 1$$
commute if and only if $\Im\langle f, g\rangle = 0$, $T_1g = T_2f$, and $[T_1, T_2] = 0$. Moreover, if $a$ and $b$ commute, then the distribution $\mu_{(a, b)}$ of $(a, b)$ is $\boxplus\boxplus$-infinitely divisible, i.e. for every $n \in \mathbb{N}$, $\mu_{(a, b)}$ can be written as an $n$-fold additive bi-free convolution of some planar probability measure.
\end{prop}

\begin{proof}
Without loss of generality, we may assume $\lambda_1 = \lambda_2 = 0$. First, we consider the commutativity between $a$ and $b$. Since $\ell(f)r(g)\Omega = r(g)\ell(f)\Omega$ and $\ell(f)r(g)(\xi_1 \otimes \cdots \otimes \xi_n) = r(g)\ell(f)(\xi_1 \otimes \cdots \otimes \xi_n)$ for all $\xi_1, \dots, \xi_n \in \mathcal{H}$ with $n \geq 1$, it follows that
\begin{equation}\label{commute1}
\ell(f)r(g) = r(g)\ell(f)\;\;\;\;\;\text{and}\;\;\;\;\;\ell(f)^*r(g)^* = r(g)^*\ell(f)^*.
\end{equation}
Note that we also have $\ell(f)r(g)^*\Omega = 0$, $r(g)^*\ell(f)\Omega = \langle f, g\rangle\Omega$, and $\ell(f)r(g)^* = r(g)^*\ell(f)$ on $\mathcal{F(H)} \ominus \mathbb{C}\Omega$. On the other hand, we have $r(g)\ell(f)^*\Omega = 0$, $\ell(f)^*r(g)\Omega = \langle g, f\rangle\Omega$, and $\ell(f)^*r(g) = r(g)\ell(f)^*$ on $\mathcal{F(H)} \ominus \mathbb{C}\Omega$, which imply that the operator
$$A = \ell(f)r(g)^* + \ell(f)^*r(g) - r(g)^*\ell(f) - r(g)\ell(f)^*$$
vanishes on $\mathcal{F(H)} \ominus \mathbb{C}\Omega$ and satisfies
\begin{equation}\label{commute2}
A\Omega = -i2(\Im\langle f,g\rangle)\Omega.
\end{equation}
Similarly, the conditions $\ell(f)\Lambda_r(T_2)\Omega = 0$, $\Lambda_r(T_2)\ell(f)\Omega = T_2f$, and $\ell(f)\Lambda_r(T_2) = \Lambda_r(T_2)\ell(f)$ on $\mathcal{F(H)} \ominus \mathbb{C}\Omega$
combining with the conditions $r(g)\Lambda_\ell(T_1)\Omega = 0$, $\Lambda_\ell(T_1)r(g)\Omega = T_1g$, and $r(g)\Lambda_\ell(T_1) = \Lambda_\ell(T_1)r(g)$ on $\mathcal{F(H)} \ominus \mathbb{C}\Omega$ imply that the operator
$$B = \ell(f)\Lambda_r(T_2) + \Lambda_\ell(T_1)r(g) - \Lambda_r(T_2)\ell(f) - r(g)\Lambda_\ell(T_1)$$
vanishes on $\mathcal{F(H)} \ominus \mathbb{C}\Omega$ and satisfies
\begin{equation}\label{commute3}
B\Omega = T_1g - T_2f.
\end{equation}
One can also see from the identities $\ell(f)^*\Lambda_r(T_2)\xi = \langle T_2\xi, f\rangle$, $\Lambda_r(T_2)\ell(f)^*\xi = 0$ for $\xi \in \mathcal{H}$, and $\ell(f)^*\Lambda_r(T_2) = \Lambda_r(T_2)\ell(f)^*$ on $\mathcal{F(H)} \ominus \mathcal{H}$, and the identities $\Lambda_\ell(T_1)r(g)^*\xi = 0$, $r(g)^*\Lambda_\ell(T_1)\xi = \langle T_1\xi, g\rangle$, and $\Lambda_\ell(T_1)r(g)^* = r(g)^*\Lambda_\ell(T_1)$ on $\mathcal{F(H)} \ominus \mathcal{H}$ that the operator
$$C = \ell(f)^*\Lambda_r(T_2) + \Lambda_\ell(T_1)r(g)^* - \Lambda_r(T_2)\ell(f)^* - r(g)^*\Lambda_\ell(T_1)$$
vanishes on $\mathcal{F(H)} \ominus \mathcal{H}$ and satisfies
\begin{equation}\label{commute4}
C\xi = \langle\xi, (T_2f - T_1g)\rangle,\;\;\;\;\;\xi \in \mathcal{H}.
\end{equation}
Finally, the conditions that $\Lambda_\ell(T_1)\Lambda_r(T_2)\xi = T_1T_2\xi$, $\Lambda_r(T_2)\Lambda_\ell(T_1)\xi = T_2T_1\xi$ for $\xi \in \mathcal{H}$, and $\Lambda_\ell(T_1)\Lambda_r(T_2) = \Lambda_r(T_2)\Lambda_\ell(T_1)$ on $\mathcal{F(H)} \ominus \mathcal{H}$ imply that the operator
$$D = \Lambda_\ell(T_1)\Lambda_r(T_2) - \Lambda_r(T_2)\Lambda_\ell(T_1)$$
vanishes on $\mathcal{F(H)} \ominus \mathcal{H}$ and satisfies
\begin{equation}\label{commute5}
D\xi = (T_1T_2 - T_2T_1)\xi,\;\;\;\;\;\xi \in \mathcal{H}.
\end{equation}
The desired result now follows from the fact that $ab - ba = A + B + C + D$ and the established identities \eqref{commute1}-\eqref{commute5}.

For the second assertion, fix $n \in \mathbb{N}$ and let
$$\mathcal{H}_n = \underbrace{\mathcal{H} \oplus \cdots \oplus \mathcal{H}}_{n\;\text{times}}.$$
Furthermore, let
$$\tilde{a} = \ell\left(\dfrac{f \oplus \cdots \oplus f}{\sqrt{n}}\right) + \ell\left(\dfrac{f \oplus \cdots \oplus f}
{\sqrt{n}}\right)^* + \Lambda_\ell(T_1 \oplus \cdots \oplus T_1)$$
and
$$\tilde{b} = r\left(\dfrac{g \oplus \cdots \oplus g}{\sqrt{n}}\right) + r\left(\dfrac{g \oplus \cdots \oplus g}
{\sqrt{n}}\right)^* + \Lambda_r(T_2 \oplus \cdots \oplus T_2).$$
Then $[\tilde{a}, \tilde{b}] = 0$ and the distribution of $(a, b)$ in $(B(\mathcal{F}(\mathcal{H})), \tau_\mathcal{H})$ is same as the distribution of $(\tilde{a}, \tilde{b})$ in $(B(\mathcal{F}(\mathcal{H}_n)), \tau_{\mathcal{H}_n})$. Note that $\tilde{a}$ and $\tilde{b}$ can be written as
$$\tilde{a} = \left[\ell\left(\dfrac{f \oplus 0 \oplus \cdots \oplus 0}{\sqrt{n}}\right) + \ell\left(\dfrac{f \oplus 0 \oplus \cdots \oplus 0}{\sqrt{n}}\right)^* + \Lambda_\ell(T_1 \oplus 0 \oplus \cdots \oplus 0)\right] +$$
$$\cdots + \left[\ell\left(\dfrac{0 \oplus \cdots \oplus 0 \oplus f}{\sqrt{n}}\right) + \ell\left(\dfrac{0 \oplus \cdots \oplus 0 \oplus f}{\sqrt{n}}\right)^* + \Lambda_\ell(0 \oplus \cdots \oplus 0 \oplus T_1)\right]$$
and
$$\tilde{b} = \left[r\left(\dfrac{g \oplus 0 \oplus \cdots \oplus 0}{\sqrt{n}}\right) + r\left(\dfrac{g \oplus 0 \oplus \cdots \oplus 0}{\sqrt{n}}\right)^* + \Lambda_r(T_2 \oplus 0 \oplus \cdots \oplus 0)\right] +$$
$$\cdots + \left[r\left(\dfrac{0 \oplus \cdots \oplus 0 \oplus g}{\sqrt{n}}\right) + r\left(\dfrac{0 \oplus \cdots \oplus 0 \oplus g}{\sqrt{n}}\right)^* + \Lambda_r(0 \oplus \cdots \oplus 0 \oplus T_2)\right].$$
Denote the $n$ summands of $\tilde{a}$ (respectively, of $\tilde{b}$) in the above summations by $\tilde{a}_1, \dots, \tilde{a}_n$ (respectively, $\tilde{b}_1, \dots, \tilde{b}_n$). Then $\{(\tilde{a}_k, \tilde{b}_k)\}_{1 \leq k \leq n}$ are bi-free and identically distributed by Remark \ref{orthogonal}. Moreover, for every $1 \leq k \leq n$, $\tilde{a}_k$ and $\tilde{b}_k$ are commuting self-adjoint random variables, thus the distribution of $(\tilde{a}_k, \tilde{b}_k)$ is a Borel probability measure on $\mathbb{R}^2$. This shows that the distribution of $(a, b)$ is the $n$-fold additive bi-free convolution of the distribution of $(\tilde{a}_k,\tilde{b}_k)$. Since $n \in \mathbb{N}$ was arbitrary, the second assertion follows.
\end{proof}

\begin{prop}\label{vanishkappa}
Following the same notations as in \emph{Proposition \ref{commuteprop}}, if $[a, b] = 0$, then the bi-free cumulants of
$(a, b)$ are given as follows:
$$\kappa_{1, 0}(a, b) = \lambda_1,\;\;\;\;\;\kappa_{0, 1}(a, b) = \lambda_2,$$
$$\kappa_{m, 0}(a, b) = \kappa_m(\ell(f)^*, \underbrace{\Lambda_\ell(T_1), \dots, \Lambda_\ell(T_1)}_{m - 2\;\mathrm{times}}, \ell(f)) = \langle T_1^{m - 2}f, f\rangle,\;\;\;\;\;m \geq 2,$$
$$\kappa_{0, n}(a, b) = \kappa_n(r(g)^*, \underbrace{\Lambda_r(T_2), \dots, \Lambda_r(T_2)}_{n - 2\;\mathrm{times}}, r(g)) = \langle T_2^{n - 2}g, g\rangle,\;\;\;\;\;n \geq 2,$$
and
$$\kappa_{m, n}(a, b) = \langle\Lambda_\ell(T_1)^{m - 1}\ell(f)\Omega, \Lambda_r(T_2)^{n - 1}r(g)\Omega\rangle =
\langle T_1^{m - 1}f, T_2^{n - 1}g\rangle$$
for $m, n \geq 1$.
\end{prop}

\begin{proof}
The equalities concerning $\kappa_1(a)$, $\kappa_1(b)$, $\kappa_m(a)$, and $\kappa_n(b)$ are known results in (one-variable) free probability theory (see \cite[Proposition 13.5]{NS06}). For the equality concerning $\kappa_{m, n}(a, b)$, we use the same setup (with $\mathcal{H}_N$ replacing $\mathcal{H}_n$ to avoid confusion with the subscript $n$ in $\kappa_{m, n}(a, b)$) as in the proof of Proposition $\ref{commuteprop}$. Observe that the random variables
$$\ell(f),\;\ell(f)^*,\;r(g),\;r(g)^*,\;\Lambda_\ell(T_1),\;\Lambda_r(T_2)$$
in $(B(\mathcal{F}(\mathcal{H})), \tau_\mathcal{H})$ have the same joint distribution as the random variables
$$\ell\left(\dfrac{f \oplus \cdots \oplus f}{\sqrt{N}}\right),\;\ell\left(\dfrac{f \oplus \cdots \oplus f}{\sqrt{N}}\right)^*,\;r\left(\dfrac{g \oplus \cdots \oplus g}{\sqrt{N}}\right),\;r\left(\dfrac{g \oplus \cdots \oplus g}{\sqrt{N}}\right)^*,$$
$$\Lambda_\ell(T_1 \oplus \cdots \oplus T_1),\;\Lambda_r(T_2 \oplus \cdots \oplus T_2)$$
in $(B(\mathcal{F}(\mathcal{H}_N)), \tau_{\mathcal{H}_N})$. Moreover, the latter random variables are the sums of $N$ random variables, where the summands have the same joint distribution as the random variables
$$\dfrac{1}{\sqrt{N}}\ell(f),\;\dfrac{1}{\sqrt{N}}\ell(f)^*,\;\dfrac{1}{\sqrt{N}}r(g),\;\dfrac{1}{\sqrt{N}}r(g)^*,\;\Lambda_\ell(T_1),\;\Lambda_r(T_2)$$
in $(B(\mathcal{F}(\mathcal{H})), \tau_\mathcal{H})$. Expanding $\kappa_{m, n}(a, b)$ using the multilinearity of the cumulants, we see that each summand is of the form $\kappa_{m + n}(c_1, \dots, c_m, d_1, \dots, d_n)$, where
$$c_i \in \{\ell(f), \ell(f)^*,\Lambda_\ell(T_1)\}\;\;\;\;\;\mathrm{and}\;\;\;\;\;d_j\in \{r(g), r(g)^*, \Lambda_r(T_2)\}.$$
By Theorem \ref{BFlimthm}, we have
$$\kappa_{m + n}(c_1, \dots, c_m, d_1, \dots, d_n) = \lim_{N \rightarrow \infty}N\cdot\langle\tilde{c}_1\cdots
\tilde{c}_m\tilde{d}_1\cdots\tilde{d}_n\Omega, \Omega\rangle$$
where
$$\tilde{c}_i \in \left\{\dfrac{\ell(f)}{\sqrt{N}}, \dfrac{\ell(f)^*}{\sqrt{N}}, \Lambda_\ell(T_1)\right\}\;\;\;\;\;\mathrm{and}\;\;\;\;\;\tilde{d}_j \in \left\{\dfrac{r(g)}{\sqrt{N}}, \dfrac{r(g)^*}{\sqrt{N}}, \Lambda_r(T_2)\right\}$$
in correspondence to $c_i$ and $d_j$. Since $m, n \geq 1$, and by the definitions of how creation, annihilation, and gauge operators act on full Fock spaces, the only non-vanishing summand of $\kappa_{m, n}(a, b)$ is
$$\kappa_{m + n}(\ell(f)^*, \underbrace{\Lambda_\ell(T_1), \dots, \Lambda_\ell(T_1)}_{m - 1\;\text{times}}, \underbrace{\Lambda_r(T_2), \dots, \Lambda_r(T_2)}_{n - 1\;\text{times}}, r(g)),$$
which is equal to
$$\lim_{N \rightarrow \infty}N\cdot\left\langle\dfrac{1}{\sqrt{N}}\ell(f)^*\Lambda_\ell(T_1)^{m - 1}\Lambda_r(T_2)^{n - 1}\dfrac{1}{\sqrt{N}}r(g)\Omega, \Omega\right\rangle = \langle T_1^{m - 1}f, T_2^{n - 1}g\rangle$$
by the hypothesis that $a$ and $b$ commute.
\end{proof}

\subsection{Bi-free L\'{e}vy-Hin\v{c}in representations}

The next lemma establishes the relations between the bi-free $\mathcal{R}$-transforms of a two-faced pair of commuting self-adjoint random variables in some $C^*$-probability space and the corresponding free $\mathcal{R}$-transforms.

\begin{lem}\label{project}
For a two-faced pair $(a,b)$ of commuting self-adjoint random variables in a $C^*$-probability space $(\mathcal{A}, \varphi)$, we have
$$\mathcal{R}_{(a, b)}(z, 0) = z\mathcal{R}_a(z)\;\;\;\;\;\mathrm{and}\;\;\;\;\;\mathcal{R}_{(a, b)}(0, w) = w\mathcal{R}_b(w)$$
for $z$ and $w$ in a small neighbourhood of $0$.
\end{lem}

\begin{proof}
By the representation of $\mathcal{R}_{(a, b)}(z, w)$ given in Theorem \ref{bifreeR}, it suffices to show that
$$\lim_{w \rightarrow 0}\dfrac{zw}{G_{(a, b)}(K_a(z), K_b(w))} = 1$$
for small $z \neq 0$. For $(z, w)$ in a deleted neighbourhood of $(0, 0)$, we have
$$G_{(a, b)}(z^{-1}, w^{-1}) = w\varphi((z^{-1} - a)^{-1}(1 - wb)^{-1}),$$
thus
$$\dfrac{zw}{G_{(a, b)}(K_a(z), K_b(w))} = \dfrac{z(w\mathcal{R}_b(w) + 1)}{\varphi((K_a(z) - a)^{-1}(1 - b/K_b(w))^{-1})}.$$
Since $\lim_{w \rightarrow 0}1/K_b(w) = 0$ and $\lim_{w \rightarrow 0}w\mathcal{R}_b(w) = 0$, it follows that
$$\lim_{w \rightarrow 0}\dfrac{zw}{G_{(a, b)}(K_a(z), K_b(w))} = \dfrac{z}{\varphi((K_a(z) - a)^{-1})} = \dfrac{z}{G_a(K_a(z))} = 1$$
as desired. The second equality can be shown in a similar way.
\end{proof}

Recall from \cite[Theorem 8.6]{BV92} that a compactly supported probability measure $\nu$ on $\mathbb{R}$ is $\boxplus$-infinitely divisible if and only if its $\mathcal{R}$-transform is of the form
\begin{equation}\label{LHformulaR}
\mathcal{R}_\nu(z) = \kappa_1^\nu + \int_\mathbb{R}\dfrac{z}{1 - zs}\;d\rho_\nu(s),
\end{equation}
where $\rho_\nu$, called the free L\'{e}vy measure of $\nu$, is a finite positive Borel measure on $\mathbb{R}$. The integral representation \eqref{LHformulaR} is usually referred to as the free L\'{e}vy-Hin\v{c}in representation for $\nu$.

Denote by $\mathbb{C}_0[s, t]$ the algebra of polynomials in $\mathbb{C}[s, t]$ with vanishing constant term, i.e. $p(0, 0) = 0$ if $p \in \mathbb{C}_0[s, t]$. As indicated in \eqref{sesquilinear}, given a real $2$-sequence $R = \{R_{m, n}\}_{(m, n) \in \mathbb{Z}_+^2 \backslash \{(0, 0)\}}$, one can equip $\mathbb{C}_0[s, t]$ with a sesquilinear form $[\cdot, \cdot]_R$ defined by
\begin{equation}\label{sesquilinear2}
[s^{m_1}t^{n_1}, s^{m_2}t^{n_2}]_R = R_{m_1 + m_2, n_1 + n_2}
\end{equation}
for $(m_1, n_1), (m_2, n_2) \neq (0,0)$.

\begin{defn}\label{conditiondef}
A real $2$-sequence $R = \{R_{m, n}\}_{(m, n)\in\mathbb{Z}_+^2 \backslash \{(0, 0)\}}$ is said to be \textit{conditionally positive semi-definite} if
$$[p, p]_R \geq 0$$
for all $p \in \mathbb{C}_0[s,t]$, where $[\cdot, \cdot]_R$ is the sesquilinear form defined above. The $2$-sequence $R$ is said to be \textit{conditionally bounded} if there exists a finite number $L > 0$ such that
$$|[s^mt^np, p]_R| \leq L^{m+n}\cdot[p,p]_R$$
for all $p \in \mathbb{C}_0[s, t]$ and $m, n \geq 0$.
\end{defn}

Note that conditional positive semi-definiteness and conditional boundedness of $R$ do not depend on the values of $R_{1, 0}$ and $R_{0, 1}$ as they do not appear in any summand of $[p, p]_R$ and $[s^mt^np, p]_R$ for any $p \in \mathbb{C}_0[s, t]$. Moreover, note that conditional positive semi-definiteness of $R$ together with the assumption $R_{2, 0}=0$ (resp. $R_{0, 2} = 0$) yields that $R_{m, n} = 0$ if $m + n \geq 2$ and $m \geq 1$ (resp. $R_{m, n} = 0$ if $m + n \geq 2$ and $n \geq 1$) by Cauchy-Schwarz inequality.

We are now ready to characterize $\boxplus\boxplus$-infinitely divisible distributions and provide their bi-free L\'{e}vy-Hin\v{c}in representations.

\begin{thm}\label{IDCS}
Let $\mu$ be a compactly supported planar probability measure and let $\kappa = \{\kappa_{m, n}^\mu\}_{(m, n) \in \mathbb{Z}_+^2 \backslash \{(0, 0)\}}$ be the bi-free cumulants of $\mu$ such that $\kappa_{2, 0}^\mu, \kappa_{0, 2}^\mu > 0$. Then the following statements are equivalent.
\begin{enumerate}[$(1)$]
\item The measure $\mu$ is $\boxplus\boxplus$-infinitely divisible.

\item The $2$-sequence $\kappa$ is conditionally positive semi-definite and conditionally bounded in the sense defined in \emph{Definition \ref{conditiondef}}.

\item There exist finite positive Borel measures $\rho_1$ and $\rho_2$ with compact supports on $\mathbb{R}^2$ and
a finite Borel measure $\rho$ on $\mathbb{R}^2$ satisfying the relations
\begin{equation}\label{atomineq}
|\rho(\{(0, 0)\})|^2 \leq \rho_1(\{(0, 0)\})\rho_2(\{(0, 0)\}),
\end{equation}
\begin{equation}\label{relation}
td\rho_1(s, t) = sd\rho(s, t),\;\;\;\;\;\text{and}\;\;\;\;\;sd\rho_2(s, t) = td\rho(s, t)
\end{equation} such that
\begin{equation}\label{LHformula}
\mathcal{R}_\mu(z, w) = z\mathcal{R}_1(z) + w\mathcal{R}_2(w) + \int_{\mathbb{R}^2}\dfrac{zw}{(1 - zs)(1 - wt)}\;d\rho(s, t)
\end{equation}
holds for $(z, w)$ in a neighbourhood of $(0, 0)$, where
$$\mathcal{R}_1(z) = \kappa_{1, 0}^\mu + \int_{\mathbb{R}^2}\dfrac{z}{1 - zs}\;d\rho_1(s,t)\;\;\;\;\;\text{and}
\;\;\;\;\;\mathcal{R}_2(w) = \kappa_{0, 1}^\mu + \int_{\mathbb{R}^2}\dfrac{w}{1 - wt}\;d\rho_2(s, t).$$
\end{enumerate}
If assertions \textnormal{(1)-(3)} hold, then $\mathcal{R}_\mu$ extends analytically to $(\mathbb{C} \backslash \mathbb{R})^2$ via the formula \eqref{LHformula}, and the functions $\mathcal{R}_1$ and $\mathcal{R}_2$ in assertion $(3)$ are the free $\mathcal{R}$-transforms of $\boxplus$-infinitely divisible distributions.
\end{thm}

\begin{proof}
Suppose first that assertion $(1)$ holds, i.e. for every $N \in \mathbb{N}$, there are commuting self-adjoint random variables $a_N$ and $b_N$ in some $C^*$-probability space $(\mathcal{A}_N, \varphi_N)$ with distribution $\mu_N$ such that $\underbrace{\mu_N \boxplus\boxplus \cdots \boxplus\boxplus \mu_N}_{N\;\mathrm{times}} = \mu$. Then, by Theorem \ref{BFlimthm}, the bi-free cumulants of $\mu$ are given by
$$\kappa_{m, n}^\mu = \lim_{N \rightarrow \infty}N\cdot\varphi_N(a_N^mb_N^n),\;\;\;\;\;(m, n) \in \mathbb{Z}_+^2 \backslash \{(0, 0)\}.$$
Observe that for any polynomial $p = p(s, t) = \sum_{j = 1}^\ell c_js^{m_j}t^{n_j} \in \mathbb{C}_0[s, t]$, we have
\begin{align*}
[p, p]_{\kappa} &= \sum_{j, k = 1}^\ell c_j\overline{c}_k\kappa_{m_j + m_k, n_j + n_k}^\mu\\
&= \sum_{j, k = 1}^\ell c_j\overline{c}_k\lim_{N \rightarrow \infty}N\cdot\varphi_N(a_N^{m_j + m_k}b_N^{n_j + n_k})\\
&= \lim_{N\to\infty}N\cdot\varphi_N\left(\sum_{j, k = 1}^\ell c_j\overline{c}_ka_N^{m_j + m_k}b_N^{n_j + n_k}\right)\\
&= \lim_{N \rightarrow \infty}N\cdot\varphi_N(p(a_N,b_N)p(a_N,b_N)^*) \geq 0,
\end{align*}
which yields the conditional positive semi-definiteness of $\kappa$. On the other hand, viewing $\mu$ as the distribution of a two-faced pair $(a, b)$ of commuting self-adjoint random variables in some $C^*$-probability space and using Lemma \ref{project}, we have for $N \in \mathbb{N}$ and for $z$ in a small deleted neighbourhood of $0$,
\begin{equation}\label{freebifreeR}
z\mathcal{R}_a(z) = \mathcal{R}_{(a, b)}(z, 0) = N\cdot\mathcal{R}_{(a_N, b_N)}(z, 0) = N\cdot z\mathcal{R}_{a_N}(z) = z\mathcal{R}_{\mu_{a_N}^{\boxplus N}}(z),
\end{equation}
where $\mu_{a_N}^{\boxplus N} = \underbrace{\mu_{a_N} \boxplus \cdots \boxplus\mu_{a_N}}_{N\;\mathrm{times}}$. This shows that the distribution of $a$ is $\boxplus$-infinitely divisible and $\mu_{a_N}^{\boxplus N} = \mu_a$. We thus conclude from \cite[Lemma 8.5]{BV92} the existence of a finite number $L > 0$ such that $\mathrm{supp}(\mu_{a_N}) \subset [-L, L]$ for all $N$. Similarly, the distribution of $b$ is $\boxplus$-infinitely
divisible, $\mu_{b_N}^{\boxplus N} = \mu_b$, and $\mathrm{supp}(\mu_{b_N}) \subset [-L, L]$ for all $N$ if $L$ is
chosen large enough. This implies that the supports of the distributions $(\mu_{(a_N, b_N)})_{N \in \mathbb{N}}$ are uniformly bounded. Indeed, for all $c > L$ and $m, n \in \mathbb{N}$, we have
\begin{align*}
c^{2m}\mu_{(a_N, b_N)}(\{(s, t) : |s| \geq c\}) & \leq \int_{\mathbb{R}^2}s^{2m}\;d\mu_{(a_N, b_N)}(s, t)\\
&= \varphi_N(a_N^{2m}) = \int_\mathbb{R}s^{2m}d\mu_{a_N}(s) \leq L^{2m},
\end{align*}
and similarly $\mu_{(a_N, b_N)}(\{(s, t) : |t| \geq c\}) \leq (L/c)^{2n}$.  This allows us to conclude that $\mathrm{supp}(\mu_{(a_N, b_N)}) \subset [-L, L]^2$ by observing that
 $\{(s, t) : |s|\;\text{or}\;|t| \geq c\}$ is a $\mu_{(a_N, b_N)}$-measure zero set by letting $m, n \rightarrow \infty$. Therefore, for $p \in \mathbb{C}_0[s, t]$ as before, we have
\begin{align*}
|[s^mt^np, p]_{\kappa}| &= \left|\sum_{j, k = 1}^\ell c_j\overline{c}_k\kappa_{m_j + m_k + m, n_j + n_k + n}^\mu\right|\\
&= \left|\lim_{N \rightarrow \infty}N\cdot\varphi_N(a_N^mb_N^np(a_N, b_N)p(a_N, b_N)^*)\right|\\
&\leq \lim_{N \rightarrow \infty}N\cdot\int_{\mathbb{R}^2}|s|^m|t|^n|p(s, t)|^2\;d\mu_N(s, t)\\
&\leq L^{m + n}\cdot\lim_{N \rightarrow \infty}N\cdot\int_{\mathbb{R}^2}|p(s, t)|^2\;d\mu_N(s, t) = L^{m + n}\cdot[p, p]_\kappa
\end{align*}
as desired. Hence, assertion (2) holds.

Next, we show the implication $(2) \Rightarrow (1)$. Observe that the conditional positive semi-definiteness of the bi-free cumulants $\kappa$ of $\mu$ shows that the sesquilinear form $[\cdot, \cdot]_{\kappa}$ on $\mathbb{C}_0[s, t]$ induced by $\kappa$ is non-negative. Moreover, the conditional boundedness of $\kappa$ indicates that $\mathcal{I} = \{p_0 \in \mathbb{C}_0[s, t] : [p_0,p_0]_{\kappa} = 0\}$ is an ideal of $\mathbb{C}_0[s, t]$ as
$$0 \leq [s^mt^np_0, s^mt^np_0]_{\kappa} \leq L^{2(m + n)}\cdot[p_0, p_0]_{\kappa} = 0$$
for all $p_0 \in \mathcal{I}$. Let $\mathcal{H}_0$ be the quotient vector space $\mathbb{C}_0[s, t] / \mathcal{I}$. If $h = p + \mathcal{I}$ and $k = q + \mathcal{I}$ are in $\mathcal{H}_0$, then
\begin{equation}\label{innerproduct}
\langle h, k\rangle_0 := [p, q]_{\kappa}
\end{equation}
can be verified to be a well-defined inner product on $\mathcal{H}_0$ because $\mathcal{I}$ is an ideal. Let $\mathcal{H}$ be the Hilbert space obtained by completing $\mathcal{H}_0$ with respect to the norm $\|\cdot\|$ defined by the inner product \eqref{innerproduct}. Consider now the linear transformations $T_1, T_2 : \mathcal{H}_0 \rightarrow \mathcal{H}$ defined by
$$T_1h = sp(s, t) + \mathcal{I}\;\;\;\;\;\text{and}\;\;\;\;\;T_2h = tp(s, t) + \mathcal{I}$$
for $h = p + \mathcal{I} \in \mathcal{H}_0$. Note that both $T_1$ and $T_2$ are well-defined since $\mathcal{I}$ is an ideal. Moreover, the inequality
$$\|T_1h\|^2 = \langle T_1h, T_1h\rangle_0 = [sp, sp]_{\kappa} \leq L^2\cdot\|h\|^2$$
shows that $T_1$ can be extended to a bounded linear operator on $\mathcal{H}$. On the other hand, if $h = p + \mathcal{I}$ and $k = q + \mathcal{I}$, where $p = \sum_{j = 1}^\ell c_js^{m_j}t^{n_j}, q = \sum_{j = 1}^\ell d_js^{u_j}t^{v_j} \in C_0[s, t]$, then we have
\begin{align*}
\langle T_1^*h, k\rangle_0 = \langle h, T_1k\rangle_0 &= \sum_{j, k = 1}^\ell c_j\overline{d}_k\kappa_{m_j + (u_k + 1), n_j + v_k}^\mu\\
&= \sum_{j, k = 1}^\ell c_j\overline{d}_k\kappa_{(m_j + 1) + u_k, n_j + v_k}^\mu = \langle T_1h, k\rangle_0,
\end{align*}
from which we see that $T_1$ is self-adjoint. Similarly, one can show that $T_2$ extends to a self-adjoint operator in
$B(\mathcal{H})$. To finish the proof of $(2) \Rightarrow (1)$, let us consider the operators
$$a = \ell(f) + \ell(f)^* + \Lambda_\ell(T_1) + \kappa_{1, 0}^\mu\cdot 1\;\;\;\;\;\text{and}\;\;\;\;\;b = r(g) + r(g)^* + \Lambda_r(T_2) + \kappa_{0, 1}^\mu\cdot 1,$$
where $f = s + \mathcal{I}$ and $g = t + \mathcal{I}$, in the $C^*$-probability space $(B(\mathcal{F}(\mathcal{H})), \tau_\mathcal{H})$. Clearly, $\Im\langle f, g\rangle = \Im\kappa_{1, 1}^\mu = 0$, $T_1g = st + \mathcal{I} = T_2f$, and $T_1T_2 = T_2T_1$ as
$$T_2T_1h = tsp(s, t) + \mathcal{I} = T_1T_2h$$
for $h = p + \mathcal{I}$ in the dense subspace $\mathcal{H}_0$ of $\mathcal{H}$. Hence, by Proposition \ref{commuteprop}, we see that $(a, b)$ is a two-faced pair of commuting self-adjoint random variables in $(B(\mathcal{F}(\mathcal{H})), \tau_\mathcal{H})$ whose distribution $\mu_{(a, b)}$ is $\boxplus\boxplus$-infinitely
divisible. We claim that $\mu_{(a, b)} = \mu$ by checking that all the bi-free cumulants $\kappa_{m, n}(a, b)$ of $(a, b)$ agree with the corresponding bi-free cumulants $\kappa$ of $\mu$. The claim follows directly from Proposition \ref{vanishkappa}. Indeed, for all $m, n \geq 1$, we have
$$\kappa_{m, n}^{\mu_{(a, b)}} = \kappa_{m, n}(a, b) = \langle T_1^{m-1}f, T_2^{n-1}g\rangle_0 = [s^m, t^n]_{\kappa} = \kappa_{m, n}^\mu.$$
If $m \geq 2$ and $n = 0$, then we have
$$\kappa_{m, 0}^{\mu_{(a, b)}} = \kappa_m(a) = \langle T_1^{m-2}f, f\rangle_0 = [s^{m-1}, s]_{\kappa} = \kappa_{m, 0}^\mu,$$
and similarly we have $\kappa_{0, n}^{\mu_{(a, b)}} = \kappa_{0, n}^\mu$ for all $n \geq 2$. Clearly, we also have
$$\kappa_{1, 0}^{\mu_{(a, b)}} = \kappa_1(a) = \kappa_{1, 0}^\mu\;\;\;\;\;\text{and}\;\;\;\;\;\kappa_{0, 1}^{\mu_{(a, b)}} = \kappa_1(b) = \kappa_{0, 1}^\mu.$$
Thus the two distributions $\mu_{(a, b)}$ and $\mu$ coincide, and hence assertion $(1)$ holds.

In what follows, we show that assertions $(2)$ and $(3)$ are equivalent. Suppose first that assertion $(2)$ holds. Since the $2$-sequences $\{\theta_{m, n}^{(1)}\}_{(m, n) \in \mathbb{Z}_+^2}$ and $\{\theta_{m, n}^{(2)}\}_{(m, n) \in \mathbb{Z}_+^2}$ defined by $\theta_{m, n}^{(1)} = \kappa_{m + 2, n}^\mu$ and $\theta_{m, n}^{(2)} = \kappa_{m, n + 2}^\mu$ are positive semi-definite and bounded, and the numbers $\theta_{0, 0}^{(1)}$ and $\theta_{0,0}^{(2)}$ are both positive, it follows from Theorem \ref{moment1} that there exist two finite positive Borel measures $\rho_1$ and $\rho_2$ with compact supports on $\mathbb{R}^2$ such that
\begin{equation}\label{Rmn1}
\kappa_{m + 2, n}^\mu = \int_{\mathbb{R}^2}s^mt^n\;d\rho_1(s, t),\;\;\;\;\;m, n \geq 0
\end{equation}
and
\begin{equation}\label{Rmn2}
\kappa_{m, n + 2}^\mu = \int_{\mathbb{R}^2}s^mt^n\;d\rho_2(s, t),\;\;\;\;\;m, n \geq 0.
\end{equation}
Observe that for $|z|$ and $|w|$ small enough so that $\|zT_1\| < 1$ and $\|wT_2\| < 1$, following the notations introduced in the proof of $(2) \Rightarrow (1)$, we have
\begin{align*}
\sum_{m, n \geq 1}\kappa_{m, n}^\mu z^mw^n &= zw\sum_{m, n \geq 0}\langle(zT_1)^mf, (wT_2)^ng\rangle\\
&= zw\langle(1 - zT_1)^{-1}f, (1 - wT_2)^{-1}g\rangle,
\end{align*}
where $\langle\cdot, \cdot\rangle$ is the inner product on $\mathcal{H}$. Let the spectral resolutions of $T_1$ and $T_2$ be
$$T_1 = \int_\mathbb{R}s\;dE_1(s)\;\;\;\;\;\text{and}\;\;\;\;\;T_2=\int_\mathbb{R}t\;dE_2(t).$$
Since $E_1(s)$ commutes with $E_2(t)$ for all $s$ and $t$, it follows from the spectral theorem that
\begin{equation}\label{Rmn4}
\sum_{m, n \geq 1}\kappa_{m, n}^\mu z^mw^n = \int_{\mathbb{R}^2}\dfrac{zw}{(1 - zs)(1 - wt)}\;d\rho(s, t),
\end{equation}
where $\rho(s, t) = \langle E_1(s)E_2(t)f, g\rangle$ is a finite compactly supported Borel measure on $\mathbb{R}^2$. Similarly, using \eqref{Rmn1}, \eqref{Rmn2}), and the spectral resolutions of $T_1$ and $T_2$, one can obtain
\begin{equation} \label{relation1}
z^2\cdot\int_{\mathbb{R}^2}\frac{d\rho_1(s, t)}{1 - zs} = \sum_{m \geq 2}\kappa_{m, 0}^\mu z^m = z^2\cdot\int_{\mathbb{R}}\frac{d\langle E_1(s)f, f\rangle}{1 - zs}
\end{equation}
and
\begin{equation}\label{relation3}
w^2\cdot\int_{\mathbb{R}^2}\frac{d\rho_2(s, t)}{1 - wt} = \sum_{n \geq 2}\kappa_{0, n}^\mu w^n = w^2\cdot\int_{\mathbb{R}}\frac{d\langle E_2(t)g, g\rangle}{1 - wt}
\end{equation}
provided that $|z|$ and $|w|$ are sufficiently small. Note that we can rewrite \eqref{relation1} and \eqref{relation3} as
$$\int_{\mathbb{R}^2}\frac{d\rho_1(s, t)}{z - s} = \int_\mathbb{R}\frac{d\langle E_1(s)f, f\rangle}{z - s}\;\;\;\;\;\text{and}\;\;\;\;\;\int_{\mathbb{R}^2}\frac{d\rho_2(s, t)}{w - t} = \int_\mathbb{R}\frac{d\langle E_2(t)g, g\rangle}{w - t},$$
which hold for $z$ and $w$ in a neighbourhood of infinity, and then in the whole $\mathbb{C} \backslash \mathbb{R}$ by the uniqueness of the analytic extension. Since the Cauchy transforms determine the underlying measures uniquely, we obtain
\begin{equation}\label{relation2}
\|E_1(B)f\|^2 = \rho_1(B \times \mathbb{R})\;\;\;\;\;\text{and}\;\;\;\;\;\|E_2(B)g\|^2 = \rho_2(\mathbb{R} \times B)
\end{equation}
for every Borel set $B$ of $\mathbb{R}$. On the other hand, using \eqref{relation2} and applying Cauchy-Schwarz inequality to $\rho(\{(0, 0)\}) = \langle E_1(\{0\})f, E_2(\{0\})g\rangle$, we have
\begin{equation}\label{atom}
|\rho(\{(0, 0)\})|^2 \leq \rho_1(\{0\} \times \mathbb{R})\rho_2(\mathbb{R} \times \{0\}).
\end{equation}
Observe also that \eqref{Rmn4} shows
\begin{equation}\label{Rmn3}
\kappa_{m + 1, n + 1}^\mu = \int_{\mathbb{R}^2}s^mt^n\;d\rho(s,t),\;\;\;\;\;m, n \geq 0.
\end{equation}
Thus for all $m, n \geq 0$, we have
$$\int_{\mathbb{R}^2}(s^mt^n)t\;d\rho_1(s, t) = \kappa_{m + 2, n + 1}^\mu = \int_{\mathbb{R}^2}(s^mt^n)s\;d\rho(s, t)$$
and
$$\int_{\mathbb{R}^2}(s^mt^n)s\;d\rho_2(s, t) = \kappa_{m + 1, n + 2}^\mu = \int_{\mathbb{R}^2}(s^mt^n)t\;d\rho(s, t),$$
from which, along with the Stone-Weierstrass theorem, we obtain the relations \eqref{relation}. Notice that the first relation in \eqref{relation} implies that
\begin{align*}
\rho_1(\{0\} \times \mathbb{R}') &= \lim_{\epsilon \rightarrow 0^+}\int_{\mathbb{R}_2}\chi_{\{s = 0, |t| > \epsilon\}}(s, t)\;d\rho_1(s, t)\\
&= \lim_{\epsilon \rightarrow 0^+}\int_{\mathbb{R}_2}\frac{s}{t}\chi_{\{s = 0, |t| > \epsilon\}}(s, t)\;d\rho(s, t) = 0,
\end{align*}
where $\mathbb{R}' = \mathbb{R} \backslash \{0\}$. Similarly, making use of the second relation in \eqref{relation}, one can show that $\rho_2(\mathbb{R}' \times \{0\}) = 0$, thus the inequality \eqref{atomineq} follows from \eqref{atom}. Finally, for $|z|$ small enough, we have
$$\sum_{m \geq 2}\kappa_{m, 0}^\mu z^m = z^2\sum_{m \geq 0}\kappa_{m + 2, 0}^\mu z^m = z^2\sum_{m \geq 0}\int_{\mathbb{R}^2}(sz)^m\;d\rho_1(s, t) = \int_{\mathbb{R}^2}\dfrac{z^2}{1 - zs}\;d\rho_1(s, t),$$
and similarly for $|w|$ small enough, we have
$$\sum_{n \geq 2}\kappa_{0, n}^\mu w^n = \int_{\mathbb{R}^2}\dfrac{w^2}{1 - wt}\;d\rho_2(s, t).$$
Combining the above conclusions with the characterization of $\boxplus$-infinitely divisible distributions, we conclude that assertion $(3)$ holds.

Conversely, suppose that assertion $(3)$ holds. Then one can easily see that the bi-free cumulants of $\mu$ are given by the formulas \eqref{Rmn1}, \eqref{Rmn2}, and \eqref{Rmn3}. Also in the proof of $(2) \Rightarrow (3)$, we have seen that the relations in \eqref{relation} imply that $\rho_1(\{0\} \times \mathbb{R}') = \rho_2(\mathbb{R}' \times \{0\}) = 0$, from which we obtain $\rho([\mathbb{R}' \times \{0\}] \cup [\{0\} \times \mathbb{R}']) = 0$ by \eqref{relation} again. In the following, we shall argue that the bi-free cumulants of $\mu$ can be expressed as limits of certain
integrals with $\sigma_\epsilon$, which is defined as
$$d\sigma_\epsilon(s, t) = \frac{\chi_{\{|s|, |t| > \epsilon\}}(s, t)}{st}\;d\rho(s, t) := \frac{\chi_{\Omega_\epsilon}(s, t)}{st}\;d\rho(s, t)$$
for any $\epsilon > 0$, as the representing measures and use these expressions to conclude the desired result. Clearly, the planar measure $\sigma_\epsilon$ is a finite positive measure by the relations in \eqref{relation} and the boundedness of the support of $\rho$. Since $\chi_{\Omega_\epsilon} \rightarrow 1$ a.e. on $\mathbb{R}^2 \backslash (\mathbb{R} \times \{0\})$ with respect to $\rho_1$ as $\epsilon \rightarrow 0^+$, the first relation in \eqref{relation} yields
\begin{align*}
\rho_1(\mathbb{R}^2 \backslash (\mathbb{R} \times \{0\})) &= \lim_{\epsilon \rightarrow 0^+}\int_{\mathbb{R}^2}\chi_{\Omega_\epsilon}(s, t)\;d\rho_1(s, t)\\
&= \lim_{\epsilon \rightarrow 0^+}\int_{\mathbb{R}^2}s^2\cdot\frac{\chi_{\Omega_\epsilon}(s, t)}{st}\;d\rho(s, t) = \lim_{\epsilon \rightarrow 0^+}\int_{\mathbb{R}^2}s^2\;d\sigma_\epsilon(s, t),
\end{align*}
from which we obtain
\begin{equation}\label{expression1}
\kappa_{2, 0}^\mu = \rho_1(\mathbb{R}^2) = \rho_1(\mathbb{R} \times \{0\}) + \lim_{\epsilon \rightarrow 0^+}\int_{\mathbb{R}^2}s^2\;d\sigma_\epsilon(s, t).
\end{equation}
For any $m \geq 3$, we have
$$\kappa_{m, 0}^\mu = \int_{\mathbb{R}^2}s^{m - 2}\;d\rho_1(s, t) = \lim_{\epsilon \rightarrow 0^+}\int_{\mathbb{R}^2}s^{m - 2}\chi_{\Omega_\epsilon}(s, t)\;d\rho_1(s, t),$$
by using the property that $s^{m - 2}\chi_{\Omega_\epsilon}(s, t) \rightarrow s^{m - 2}$ a.e. on $\mathbb{R}^2$ with respect to the measure $\chi_{\{s \neq 0\}}\rho_1$ as $\epsilon \rightarrow 0^+$ and applying the dominated convergence theorem in the second equality. Therefore, by the first relation in \eqref{relation} again, we obtain
\begin{equation}\label{expression2}
\int_{\mathbb{R}^2}s^{m - 2}\;d\rho_1(s, t) = \lim_{\epsilon \rightarrow 0^+}\int_{\mathbb{R}^2}s^m\;d\sigma_\epsilon(s, t),\;\;\;\;\;m \geq 3.
\end{equation}
Making use of the second relation in \eqref{relation} and similar arguments as shown above, one can show that
\begin{equation}\label{expression3}
\kappa_{0, 2}^\mu = \rho_2(\{0\} \times \mathbb{R}) + \lim_{\epsilon \rightarrow 0^+}\int_{\mathbb{R}^2}t^2\;d\sigma_\epsilon(s, t)
\end{equation}
and
\begin{equation}\label{expression4}
\kappa_{0, n}^\mu = \lim_{\epsilon \rightarrow 0^+}\int_{\mathbb{R}^2}t^n\;d\sigma_\epsilon(s, t),\;\;\;\;\;n \geq 3.
\end{equation}
Next, observe that $\chi_{\Omega_\epsilon} \rightarrow 1$ a.e. on $\mathbb{R}^2 \backslash \{(0, 0)\}$ with respect to $\rho$ as $\epsilon \rightarrow 0^+$, thus an application of the dominated convergence theorem to the positive and negative parts $\rho^+ = \chi_{\{st \geq 0\}}\rho$ and $\rho^- = \chi_{\{st < 0\}}\rho$ of the signed measure $\rho$ yields
$$\rho(\mathbb{R}^2 \backslash \{(0, 0)\}) = \lim_{\epsilon \rightarrow 0^+}\int_{\mathbb{R}^2}\chi_{\Omega_\epsilon}(s, t)\;d\rho(s, t).$$
This implies that
\begin{equation}\label{expression5}
\kappa_{1, 1}^\mu = \rho(\mathbb{R}^2) = \rho(\{(0, 0)\}) + \lim_{\epsilon \rightarrow 0^+}\int_{\mathbb{R}^2}st\;d\sigma_\epsilon(s, t).
\end{equation}
Finally, for any pair $(m, n)$ satisfying the conditions $m, n \geq 1$ and $m\cdot n \geq 2$, the function
$s^{m - 1}t^{n - 1}\chi_{\Omega_\epsilon} \rightarrow s^{m - 1}t^{n - 1}$ a.e. on $\mathbb{R}^2$ with respect to the signed measure $\chi_{\{s \cdot t \neq 0\}}\rho$ as $\epsilon \rightarrow 0^+$, and hence the same technique implies that
\begin{equation}\label{expression6}
\kappa_{m, n}^\mu = \int_{\mathbb{R}^2}s^{m - 1}t^{n - 1}\;d\rho(s, t) = \lim_{\epsilon \rightarrow 0^+}\int_{\mathbb{R}^2}s^mt^n\;d\sigma_\epsilon(s, t).
\end{equation}
The expressions for $\kappa_{m, n}^\mu$ shown in \eqref{expression1}-\eqref{expression6} allow us to conclude that for any $p = p(s, t) = \sum_{j = 1}^\ell c_js^{m_j}t^{n_j} \in \mathbb{C}_0[s, t]$ (here we consider the most general form of $p$, that is, $(m_1, n_1) = (1, 0)$ and $(m_2, n_2) =(0, 1)$),
\begin{align*}
[p, p]_\kappa &= \sum_{j, k=1}^\ell c_j\overline{c}_k\kappa_{m_j + m_k, n_j + n_k}^\mu\\
&= m(c_1, c_2) + \lim_{\epsilon \rightarrow 0^+}\left(\sum_{j, k = 1}^\ell c_j\overline{c}_k\int_{\mathbb{R}^2}s^{m_j + m_k}t^{n_j + n_k}\;d\sigma_\epsilon(s, t)\right)\\
&= m(c_1, c_2) + \lim_{\epsilon \rightarrow 0^+}\int_{\mathbb{R}^2}|p(s, t)|^2\;d\sigma_\epsilon(s, t),
\end{align*}
where
$$m(c_1, c_2) = |c_1|^2\rho_1(\mathbb{R} \times \{0\}) + 2\Re(c_1\overline{c}_2)\rho(\{(0, 0)\}) + |c_2|^2\rho_2(\{0\} \times \mathbb{R}).$$
The inequality \eqref{atomineq} shows that $m(c_1, c_2) \geq 0$ for any complex numbers $c_1$ and $c_2$, thus we establish the positivity of $[p, p]_\kappa$ by the above calculation that
\begin{equation}\label{reduced}
[p, p]_\kappa =  m(c_1, c_2) + \lim_{\epsilon \rightarrow 0^+}\int_{\mathbb{R}^2}|p(s, t)|^2\;d\sigma_\epsilon(s, t).
\end{equation}
Similarly, making use of the expressions \eqref{expression2}, \eqref{expression4}, and \eqref{expression6}, we see that for $m, n \geq 0$ with $m + n \geq 1$,
$$[s^mt^np, p]_\kappa = \sum_{j, k = 1}^\ell c_j\overline{c}_k\kappa_{m_j + m_k + m, n_j + n_k + n}^\mu = \lim_{\epsilon \rightarrow 0^+}\int_{\mathbb{R}^2}s^mt^n|p(s, t)|^2\;d\sigma_\epsilon(s, t)$$
because the index $(m_j + m_k + m, n_j + n_k + n)$ of $\kappa^\mu$ in the summand does not belong to the set $\{(1, 0), (0, 1), (1, 1), (2, 0), (0, 2)\}$. From this, along with the result in \eqref{reduced}, we deduce that $|[s^mt^np, p]_\kappa| \leq L^{m + n}[p, p]_\kappa$, where $L = \sup\{|s|, |t|: (s, t) \in \mathrm{supp}(\rho)\} < \infty$. This yields assertion $(2)$ and completes the proof.
\end{proof}

\begin{rmk}\label{product}
Suppose that $\mu$ is a compactly supported planar probability measure. If $\kappa^\mu$ is conditionally positive semi-definite and $\kappa^\mu_{0, 2} = 0$, then $\kappa^\mu_{m, n}=0$ for any $(m, n)$ with $m + n \geq 2$ and $n \geq 1$. In this case, $\mu$ is the distribution of $(a, \kappa_{0, 1}^\mu\cdot 1)$, where $a$ is a self-adjoint random variable in some $C^*$-probability space. In other words, $\mu = \nu_a \times \delta_{\kappa_{0, 1}^\mu}$, the product measure of the distribution $\nu_a$ of $a$ and $\delta_{\kappa_{0, 1}^\mu}$. If, in addition, $\kappa^\mu$ is conditionally bounded, then the $1$-sequence $\{\kappa^\mu_{m + 2, 0}\}_{m \geq 0}$ is a Hausdorff moment sequence on a bounded interval, i.e. it is a determined moment sequence with a compactly supported representing measure $\rho_1$ on $\mathbb{R}$. This shows that $\nu_a$ is $\boxplus$-infinitely divisible by \cite[Theorem 13.16]{NS06}, and the bi-free $\mathcal{R}$-transform of $\mu$ is given by
\begin{equation}\label{special1}
\mathcal{R}_\mu(z, w) = z\left(\kappa_{1, 0}^\mu + \int_{\mathbb{R}^2}\frac{z}{1 - zs}\;d\rho_1(s)\right) + \kappa_{0, 1}^\mu w.
\end{equation}
Conversely, any function of the form on the right-hand side of \eqref{special1} is the bi-free $\mathcal{R}$-transform of some compactly supported planar probability measure. Applying this observation to $\mathcal{R}_\mu / N$ for any $N \in \mathbb{N}$ yields the $\boxplus\boxplus$-infinite divisibility of $\mu$. In general, one can easily see that
the product of two compactly supported $\boxplus$-infinitely divisible measures on $\mathbb{R}$ is $\boxplus\boxplus$-infinitely divisible.
\end{rmk}

\begin{lem}\label{moment3}
Suppose that $\{\mu_N\}_{N \in \mathbb{N}}$ is a sequence of compactly supported planar probability measures with the property that for any $(m, n) \in \mathbb{Z}_+^2$, the limit
$$M_{m, n} := \lim_{N \rightarrow \infty}\int_{\mathbb{R}^2}s^mt^n\;d\mu_N(s, t)$$
exists and is a finite number. If there exists a finite number $L > 0$ such that $|M_{m, n}| \leq L^{m + n}$ for all $(m, n) \in \mathbb{Z}_+^2$, then the $2$-sequence $M := \{M_{m, n}\}_{(m, n) \in \mathbb{Z}_+^2}$ is a determined moment sequence and its representing measure is compactly supported.
\end{lem}

\begin{proof}
It is easy to verify that for every fixed numbers $m_0$ and $n_0$ in $\mathbb{N} \cup \{0\}$, the $1$-sequences $\alpha := \{M_{m, 2n_0}\}_{m \geq 0}$ and $\beta := \{M_{2m_0, n} + R_{0, n}\}_{n \geq 0}$ fulfill the sufficient condition of being a moment sequence on $\mathbb{R}$ (Hamburger moment sequence). Then using the hypothesis that $|M_{m, n}| \leq L^{m + n}$, we see that the representing measures of $\alpha$ and $\beta$ are both compactly supported, thus $\alpha$ and $\beta$ are determined. The desired result now follows from part (2) of Remark \ref{moment2} and the hypothesis that $|M_{m, n}|\leq L^{m + n}$.
\end{proof}

\begin{exmp}
Let us see some examples of $\boxplus\boxplus$-infinitely divisible distributions and their bi-free L\'{e}vy-Hin\v{c}in representations.
\begin{enumerate}[$(1)$]
\item Let $a = \ell(f) + \ell(f)^*$ and $b = r(g) + r(g)^*$ with $\Im\langle f, g\rangle = 0$ be two semi-circular random variables in the $C^*$-probability space $(B(\mathcal{F(H)}), \tau_\mathcal{H})$. Such a two-faced pair $(a, b)$ is called a bi-free Gaussian pair (see \cite[Section 7]{V14}). By Propositions \ref{commuteprop} and \ref{vanishkappa}, the only non-vanishing bi-free cumulants of $(a, b)$ are
$$\kappa_{2, 0}^\mu = \|f\|^2,\;\;\;\;\;\kappa_{0, 2}^\mu=\|g\|^2,\;\;\;\;\;\text{and}\;\;\;\;\;\kappa_{1, 1}^\mu = \langle f, g\rangle,$$
where $\mu$ denotes the distribution of $(a, b)$. Hence,
$$\mathcal{R}_\mu(z, w) = \|f\|^2z^2 + \langle f, g\rangle zw + \|g\|^2w^2,$$
and we conclude from Theorem \ref{IDCS} that $\mu$ is a $\boxplus\boxplus$-infinitely divisible distribution with
$\rho_1 = \|f\|^2\delta_{(0, 0)}$, $\rho_2 = \|g\|^2\delta_{(0, 0)}$, and $\rho = \langle f,g\rangle\delta_{(0, 0)}$.

\item Let $\lambda > 0$ and $(\alpha, \beta) \in \mathbb{R}^2$. For $N \in \mathbb{N}$, let
$$\mu_N = \left(1 - \dfrac{\lambda}{N}\right)\delta_{(0, 0)} + \dfrac{\lambda}{N}\delta_{(\alpha, \beta)}$$
and let $\{(a_{N ; k}, b_{N ; k})\}_{k = 1}^N$ be pairs of commuting self-adjoint random variables which are bi-free and identically distributed with distribution $\mu_N$ in some $C^*$-probability space $(\mathcal{A}_N, \varphi_N)$. Note that for $m, n \geq 0$ with $m + n \geq 1$, we have
$$\kappa_{m, n} := N\cdot\varphi_N(a_{N ; 1}^mb_{N ; 1}^n) = N\cdot\dfrac{\lambda}{N}\alpha^m\beta^n = \lambda\alpha^m\beta^n,$$
thus Theorem \ref{BFlimthm} implies the existence of a two-faced pair of commuting random variables $(a, b)$ in some non-commutative probability space $(\mathcal{A}, \varphi)$ such that the bi-free cumulants $\kappa_{m,n}^{(a, b)}$ of $(a, b)$ coincide with $\kappa_{m, n}$ and the mixed moments $\int_{\mathbb{R}^2}s^mt^n\;d\mu_N^{\boxplus\boxplus N}$ converge to $\varphi(a^mb^n)$ for $m, n \geq 0$ as $N \rightarrow \infty$. Let $\kappa_\pi(a, b) = \kappa_\pi(\underbrace{a, \dots, a}_{m\;\text{times}}, \underbrace{b, \dots, b}_{n\;\text{times}})$ for $\pi \in \mathrm{NC}(m + n)$ and $L = \max\{1, \lambda, |\alpha|, |\beta|\}$. Then using the
moment-cumulant formula, and the inequalities $|\kappa_{m, n}|\leq(L^2)^{m + n}$ and $\sharp\mathrm{NC}(k) \leq 4^k$ for $k \in \mathbb{N}$, we have
$$|\varphi(a^mb^n)| \leq \sum_{\pi \in \mathrm{NC}(m + n)}|\kappa_\pi(a, b)| \leq (L^2)^{m + n}\cdot\sharp\mathrm{NC}(m + n) \leq (4L^2)^{m + n}.$$
By Lemma \ref{moment3}, we see that the distribution of the pair $(a, b)$ is some compactly supported planar probability measure $\mu$ and $\kappa_{m, n} = \kappa_{m, n}^\mu$. Finally, a simple calculation shows that
\begin{align*}
\mathcal{R}_\mu(z, w) &= \sum_{\substack{m, n \geq 0\\m + n \geq 1}}\lambda(\alpha z)^m(\beta w)^n\\
& = \lambda z\left(\alpha + \frac{\alpha^2z}{1 - \alpha z}\right) + \lambda w\left(\beta + \frac{\beta^2w}{1 - \beta w}\right) + \dfrac{\lambda\alpha\beta zw}{(1 - \alpha z)(1 - \beta w)},
\end{align*}
from which we conclude the $\boxplus\boxplus$-infinite divisibility of $\mu$ by Theorem \ref{IDCS} with $\rho_1 = \lambda s^2\delta_{(\alpha, \beta)}$, $\rho_2 = \lambda t^2\delta_{(\alpha, \beta)}$, and $\rho = \lambda st\delta_{(\alpha, \beta)}$. Observe that the function
$$\mathcal{R}(z) = \lambda\left(\alpha + \frac{\alpha^2z}{1 - \alpha z}\right)$$
is the free $\mathcal{R}$-transform of the free Poisson distribution with rate $\lambda$ and jump size $\alpha$ (see \cite[Proposition 12.11]{NS06}). We call $\mu$ the \textit{bi-free Poisson distribution} with rate $\lambda$ and jump size $(\alpha,\beta)$.

\item For any $\lambda > 0$ and compactly supported planar probability measure $\nu$, consider the distribution
$$\mu_N = \left(1 - \frac{\lambda}{N}\right)\delta_{(0, 0)} + \frac{\lambda}{N}\nu.$$
For any $(m, n) \in \mathbb{Z}_+^2$, observe that
$$\kappa_{m, n} := N\int_{\mathbb{R}^2}s^mt^n\;d\mu_N(s, t) = \lambda\int_{\mathbb{R}^2}s^mt^n\;d\nu(s, t).$$
Using similar arguments shown in $(2)$ of this example and the fact that $|\kappa_{m, n}| \leq (4L^2)^{m + n}$, where $L = \max\{1, \lambda, |s|, |t| : (s, t) \in \mathrm{supp}(\nu)\}$, one can show the existence of a compactly supported planar probability measure $\mu$ such that $\kappa_{m, n}^\mu = \kappa_{m, n}$. A simple calculation then show that
\begin{align*}
\mathcal{R}_\mu(z, w) &= \sum_{m \geq 1}\kappa_{m, 0}^\mu z^m + \sum_{n \geq 1}\kappa_{0, n}^\mu w^n + \sum_{m, n \geq 1}\kappa_{m, n}^\mu z^mw^n\\
&= z\mathcal{R}_1(z) + w\mathcal{R}_2(w) + \int_{\mathbb{R}^2}\frac{zw}{(1 - zs)(1 - wt)}\;d\rho(s, t)
\end{align*}
holds in a neighbourhood of $(0, 0)$, where
$$\mathcal{R}_1(z) = \kappa_{1, 0}^\mu + \int_{\mathbb{R}^2}\frac{z}{1 - zs}\;\lambda s^2d\nu(s, t) = \int_{\mathbb{R}}\frac{\lambda s}{1 - zs}\;d\nu^{(1)}(s),$$
$$\mathcal{R}_2(w) = \kappa_{0, 1}^\mu + \int_{\mathbb{R}^2}\frac{w}{1 - wt}\;\lambda t^2d\nu(s, t) = \int_{\mathbb{R}}\frac{\lambda t}{1 - wt}\;d\nu^{(2)}(t),$$
and $d\rho(s, t) = \lambda std\nu(s, t)$ (here, the probability measures $\nu^{(1)}$ and $\nu^{(2)}$ are the marginal laws of $\nu$, that is, $\nu^{(1)}(B)=\nu(B\times\mathbb{R})$ and $\nu^{(2)}(B)=\nu(\mathbb{R}\times B)$ for any Borel set $B$ of $\mathbb{R}$). This yields the $\boxplus\boxplus$-infinite divisibility of $\mu$ by Theorem
\ref{IDCS}. Note that the function $\mathcal{R}_j$, $j = 1, 2$, is the free $\mathcal{R}$-transform of the compound free Poisson distribution with rate $\lambda$ and jump distribution $\nu^{(j)}$ (see \cite[Proposition 12.15]{NS06}). We call $\mu$ the \textit{compound bi-free Poisson distribution} with rate $\lambda$ and jump distribution $\nu$. Finally, we mention another relation between $\nu$ and the limiting distribution $\mu$: $\mathcal{R}_\mu(z, w) = \lambda[G_\nu(1/z, 1/w) - 1]$ for $(z, w) \in (\mathbb{C} \backslash \mathbb{R})^2$. To obtain this relation, simply observe that
\begin{align*}
\mathcal{R}_\mu(z, w) &= \lambda\left(-1 + \sum_{m, n \geq 0}\int_{\mathbb{R}^2}(zs)^m(wt)^n\;d\nu(s, t)\right)\\
&= \lambda\left(-1 + \int_{\mathbb{R}^2}\frac{d\nu(s, t)}{(1 - zs)(1 - wt)}\right),
\end{align*}
which holds in a neighbourhood of $(0, 0)$, and then extends analytically to $(\mathbb{C} \backslash \mathbb{R})^2$.
\end{enumerate}
\end{exmp}

\section{Bi-free L\'{e}vy processes}

In classical probability theory, there is an important class of stochastic processes, called L\'{e}vy processes, where each process has independent and stationary increments. The non-commutative analogues of these processes are called free L\'{e}vy processes, first studied by Biane. We refer the reader to \cite{B98} for details. In this section, we shall discuss the relation between $\boxplus\boxplus$-infinitely divisible distributions and stationary processes with bi-free increments.

\begin{defn}\label{bifreeLevydef}
A \textit{bi-free L\'{e}vy process} $(Z_t)_{t \geq 0}$ is a family of pairs of commuting self-adjoint random variables in some $C^*$-probability space, that is, $Z_t = (X_t, Y_t)$ where $X_t = X_t^*$, $Y_t = Y_t^*$, and $[X_t, Y_t] = 0$, with the following properties:
\begin{enumerate}[$(1)$]
\item $Z_0 = (0, 0)$;

\item for any set of times $0 \leq t_0 < t_1 < \cdots < t_n$, the increments
$$Z_{t_0},\;Z_{t_1} - Z_{t_0},\;Z_{t_2} - Z_{t_1}, \dots, Z_{t_n} - Z_{t_{n - 1}}$$
are bi-freely independent, where $Z_t - Z_s := (X_t - X_s, Y_t - Y_s)$;

\item for all $s$ and $t$ in $[0, \infty)$, the distribution of $Z_{s + t} - Z_s$ depends only on $t$;

\item the distribution of $Z_t$ tends to $\delta_{(0, 0)}$ weakly as $t \rightarrow 0^+$.
\end{enumerate}
\end{defn}

We have the following relation between $\boxplus\boxplus$-infinitely divisible distributions and bi-free L\'{e}vy processes.

\begin{thm}\begin{enumerate}[$(1)$]
\item Let $(Z_t)_{t \geq 0}$ be a bi-free L\'{e}vy process and let $\mu_t$ be the distribution of $Z_t$, $t \geq 0$. Then $\mu_1$ is $\boxplus\boxplus$-infinitely divisible and for any $T > 0$, the distributions $(\mu_t)_{0 \leq t \leq T}$ have uniformly bounded supports. Moreover, the family $(\mu_t)_{t \geq 0}$ satisfies the additive bi-free semigroup property
\begin{equation}\label{semigroup2}
\mu_s \boxplus\boxplus \mu_t = \mu_{s + t},\;\;\;\;\;s, t \geq 0,
\end{equation}
and for any $t \geq 0$, the identity
\begin{equation}\label{semigroup3}
\mathcal{R}_{\mu_t}(z, w) = t\mathcal{R}_{\mu_1}(z, w)
\end{equation}
holds in a neighbourhood of $(0, 0)$.

\item For any $\boxplus\boxplus$-infinitely divisible compactly supported planar measure $\mu$, there exists a bi-free
L\'{e}vy process $(Z_t)_{t \geq 0}$ such that the distribution of $Z_1$ is equal to $\mu$.
\end{enumerate}
\end{thm}

\begin{proof}
First, we prove assertion (2). Using the Hilbert space $\mathcal{H}$, the vectors $f, g$ in $\mathcal{H}$, and the operators
$$a = \ell(f) + \ell(f)^* + \Lambda_\ell(T_1) + \kappa_{1, 0}^\mu\;\;\;\;\;\mathrm{and}\;\;\;\;\;b = r(g) + r(g)^* + \Lambda_r(T_2) + \kappa_{0,1}^\mu$$
in $B(\mathcal{H})$ constructed in the proof of Theorem \ref{IDCS}, the bi-free cumulants of the pair of random variables $(a, b)$ coincide with the corresponding bi-free cumulants of $\mu$. Let $\mathcal{K} = L^2(\mathbb{R}_+, dx) \otimes \mathcal{H}$, where $\mathbb{R}_+ = [0,\infty)$. For any Borel set $I$ in $\mathbb{R}_+$, denote by $\chi_I$ the characteristic function of $I$ and $M_I$ the multiplication operator by $\chi_I$ in $B(L^2(\mathbb{R}_+, dx))$. Furthermore, let
$$f_I = \chi_I \otimes f,\;\;\;\;\;g_I = \chi_I \otimes g,\;\;\;\;\;A_I = M _I \otimes T_1,\;\;\;\;\;\text{and}\;\;\;\;\;B_I = M_I \otimes T_2.$$
Consider now the family $(Z_t)_{t \geq 0} = ((X_t, Y_t))_{t \geq 0}$, where
$$X_t = \ell(f_{[0, t]}) + \ell(f_{[0, t]})^* + \Lambda_\ell(A_{[0, t]}) + t\cdot\kappa_{1, 0}^\mu$$
and
$$Y_t = r(g_{[0, t]}) + r(g_{[0, t]})^* + \Lambda_r(B_{[0, t]}) + t\cdot\kappa_{0,1}^\mu$$
are commuting self-adjoint random variables in the $C^*$-probability space $(B(\mathcal{F(K)}), \tau_\mathcal{K})$ (the property $[X_t, Y_t]=0$ follows from Proposition \ref{commuteprop}). Clearly, conditions (1) and (2) in Definition \ref{bifreeLevydef} hold for the family $(Z_t)_{t \geq 0}$ constructed above by the fact that $\{L^2((t_j, t_{j + 1}], dx) \otimes \mathcal{H}\}_{0 \leq j \leq n - 1}$ are pairwise orthogonal Hilbert spaces and Remark \ref{orthogonal}. Moreover, Proposition \ref{vanishkappa} shows that the bi-free cumulants of $Z_{s + t} - Z_s$ are given by
\begin{align*}
\kappa_{m, n} &= \langle(M_{(s, s + t]} \otimes T_1)^{m - 1}(\chi_{(s, s + t]} \otimes f), (M_{(s, s + t]} \otimes T_2)^{n - 1}(\chi_{(s, s + t]} \otimes g)\rangle\\
&= \langle\chi_{(s, s + t]} \otimes T_1^{m - 1}f, \chi_{(s, s + t]} \otimes T_2^{n - 1}g\rangle\\
&= t\langle T_1^{m - 1}f, T_2^{n - 1}g\rangle = t\kappa_{m, n}^{Z_1}
\end{align*}
for $m, n\geq 1$, where $\kappa_{m, n}^{Z_1}$ denotes the bi-free cumulants of $Z_1$, and
\begin{align*}
\kappa_{m, 0} &= \langle(M_{(s, s + t]} \otimes T_1)^{m - 2}(\chi_{(s, s + t]} \otimes f), \chi_{(s, s + t]} \otimes f\rangle\\
& = \langle\chi_{(s, s + t]} \otimes T_1^{m - 2}f, \chi_{(s, s + t]} \otimes f\rangle = t\kappa_{m, 0}^{Z_1}
\end{align*}
for $m \geq 2$. Similarly, we have $\kappa_{0, n} = t\kappa_{0, n}^{Z_1}$ for $n \geq 2$. By Proposition \ref{vanishkappa} again, we obtain $\kappa_{1, 0} = t\kappa_{1, 0}^{Z_1}$ and $\kappa_{0, 1} = t\kappa_{0, 1}^{Z_1}$. Since the bi-free cumulants of $Z_{s + t} - Z_s$ depends only on $t$, so does the distribution of $Z_{s + t} - Z_s$. Note that the above calculations also show that the bi-free cumulants of $Z_1$ coincide with the corresponding bi-free cumulants of $\mu$ and the identity \eqref{semigroup3} holds for the family $(\mu_t)_{t \geq 0}$. To finish the proof, it remains to show that $\mu_t \rightarrow \delta_{(0, 0)}$ weakly as $t \rightarrow 0^+$. Observing that $\sup_{0 \leq t \leq 1}\{\|X_t\|, \|Y_t\|\} < \infty$, it is equivalent to showing that the mixed moments of $\mu_t$ converge to
$0$ as $t \rightarrow 0^+$ because the supports of $\mu_t$, $0 \leq t \leq 1$, are uniformly bounded. Using the fact that $\kappa_{m, n}^{Z_t} = t\kappa_{m, n}^{Z_1} \rightarrow 0$ as $t \rightarrow 0^+$ for any $m, n \geq 0$ with $m + n \geq 1$ and the existence of universal polynomials on the relations of bi-free cumulants and mixed moments of planar probability distributions  (see \cite[Section 5]{V14}), the claim holds. Hence, the family $(Z_t)_{t \geq 0}$ constructed above is a bi-free L\'{e}vy process with distributions $(\mu_t)_{t \geq 0}$.

Next, we prove assertion (1). Let $(Z_t)_{t \geq 0}$ be a bi-free L\'evy process with distributions $(\mu_t)_{t \geq 0}$. Then $Z_s - Z_0$ and $Z_{s + t} - Z_s$ are bi-free, and $(Z_s - Z_0) + (Z_{s + t} - Z_s) = Z_{s + t}$. This shows that $\mu_s \boxplus\boxplus \mu_t = \mu_{s + t}$ for $s, t \geq 0$. By the semigroup property, $\mu_1$ is $\boxplus\boxplus$-infinitely divisible, thus by assertion (2) there exists an additive bi-free semigroup $(\nu_t)_{t \geq 0}$ of planar probability distributions such that $\nu_1 = \mu_1$, $\nu_t \rightarrow \delta_{(0, 0)}$ weakly as $t \rightarrow 0^+$, and $\mathcal{R}_{\nu_t} = t\mathcal{R}_{\mu_1}$. By the semigroup property, we have
$$\underbrace{\mu_{1/q} \boxplus\boxplus \cdots \boxplus\boxplus \mu_{1/q}}_{q\;\text{times}} = \underbrace{\nu_{1/q} \boxplus\boxplus \cdots \boxplus\boxplus \nu_{1/q}}_{q\;\text{times}}$$
for all $q \in \mathbb{N}$, thus $q\kappa_{m, n}^{\mu_{1/q}} = q\kappa_{m, n}^{\nu_{1/q}}$ for all $m, n \geq 0$ with $m + n \geq 1$ by the additivity of bi-free cumulants and the fact that mixed bi-free cumulants vanish. This shows $\mu_{1/q} = \nu_{1/q}$ for all $q \in \mathbb{N}$, and thus
$$\mu_{p/q} = \underbrace{\mu_{1/q} \boxplus\boxplus \cdots \boxplus\boxplus \mu_{1/q}}_{p\;\text{times}} = \underbrace{\nu_{1/q} \boxplus\boxplus \cdots \boxplus\boxplus \nu_{1/q}}_{p\;\text{times}} = \nu_{p/q}$$
for all $p/q \in \mathbb{Q} \cap [0, \infty)$. By the semigroup property again, continuity at $0$ implies continuity at any $t$, which yields $\mu_t = \nu_t$ for all $t \geq 0$. Hence, $\mathcal{R}_{\mu_t} = t\mathcal{R}_{\mu_1}$ as claimed. This finishes the proof.
\end{proof}

\section{Additive bi-free convolution semigroups}

Let $\mu$ be a compactly supported probability measure on $\mathbb{R}$. Then there exists an additive free convolution semigroup $(\mu_t)_{t \geq 1}$ with $\mu_1 = \mu$, where the existence of $\mu_t$ for large $t$ was shown by Bercovici and Voiculescu in \cite{BV95}, and later extended to all $t \geq 1$ by Nica and Speicher in \cite{NS96}. In the bi-free setting, we will use the method of Nica and Speicher to show the existence of the additive bi-free convolution semigroup generated by a compactly supported probability measure on $\mathbb{R}^2$. Let us first recall some definitions and results regarding free compressions. We refer the reader to \cite[Lecture 14]{NS06} for details.

\begin{defn}
Let $(\mathcal{A},\varphi)$ be a non-commutative probability space and $p\in\mathcal{A}$ a projection (i.e. $p^2 = p$) such that $\varphi(p) \neq 0$, then we have the \textit{compression} $(p\mathcal{A}p, \varphi^{p\mathcal{A}p})$, where
$$p\mathcal{A}p = \{pap : a \in \mathcal{A}\}$$
and
$$\varphi^{p\mathcal{A}p}(\cdot) = \dfrac{1}{\varphi(p)}\varphi(\cdot)$$
restricted to $p\mathcal{A}p$. The compression is also a non-commutative probability space with unit element $p = p\cdot 1\cdot p$. Moreover, $(\kappa_n)_{n \geq 1}$ will denote the free cumulants corresponding to $\varphi$ and $(\kappa_n^{p\mathcal{A}p})_{n\geq 1}$ will denote the free cumulants corresponding to $\varphi^{p\mathcal{A}p}$.
\end{defn}

Suppose that $(\mathcal{A}, \varphi)$ is a non-commutative probability space and $p, a_1, \dots, a_m \in \mathcal{A}$
such that $p$ is a projection with $\varphi(p) \neq 0$ and $p$ is free from $\{a_1, \dots, a_m\}$. Then recall from \cite[Theorem 14.10]{NS06} that we have
$$\kappa_n^{p\mathcal{A}p}(pa_{i(1)}p, \dots, pa_{i(n)}p) = \dfrac{1}{\varphi(p)}\kappa_n(\varphi(p)a_{i(1)}, \dots, \varphi(p)a_{i(n)})$$
for all $n \geq 1$ and all $1 \leq i(1), \dots, i(n) \leq m$. In particular, if $(\mathcal{A}, \varphi)$ is a $C^*$-probability space, $p, a_1, \dots, a_r, b_1, \dots, b_r \in \mathcal{A}$ are self-adjoint random variables such that $\varphi(p) \neq 0$, $p$ is free from $\{a_1, \dots, a_r, b_1, \dots, b_r\}$, and $[a_i, b_j] = 0$ for all $1 \leq i, j \leq r$, then we have
\begin{align}
\label{KappaAp}
\begin{split}
&\kappa_{m + n}^{p\mathcal{A}p}(pa_{i(1)}p, \dots, pa_{i(m)}p, pb_{j(1)}p, \dots, pb_{j(n)}p)\\
&= \dfrac{1}{\varphi(p)}\kappa_{m + n}(\varphi(p)a_{i(1)}, \dots, \varphi(p)a_{i(m)}, \varphi(p)b_{j(1)},\cdots,\varphi(p)b_{j(n)})
\end{split}
\end{align}
for all $m, n \geq 0$ with $m + n \geq 1$ and all $1 \leq i(1), \dots, i(m), j(1), \dots, j(n) \leq r$.

We shall use the following result to show the existence of the additive bi-free convolution semigroups.

\begin{lem}\label{semigrouplem}
Given random variables $a_1$, $a_2$, and $p$ in some $C^*$-probability space $(\mathcal{A}, \varphi)$ such that $a_1 = a_1^*$, $a_2 = a_2^*$, $[a_1, a_2] = 0$, and $p$ is a projection free from $\{a_1, a_2\}$, there exists a compactly supported probability measure $\mu$ on $\mathbb{R}^2$ such that
$$\kappa_{m, n}^\mu = \kappa_{m + n}^{p\mathcal{A}p}(\underbrace{pa_1p, \dots, pa_1p}_{m\;\text{times}}, \underbrace{pa_2p, \dots, pa_2p}_{n\;\text{times}})$$
for all $m, n \geq 0$ with $m + n \geq 1$.
\end{lem}

\begin{proof}
Let $i_1, \dots, i_n \in \{1, 2\}$ and let $\sigma \in S_n$ be a permutation. Denote $pa_jp$ by $x_j$ for $j = 1, 2$, then by the moment-cumulant formula and \cite[Theorem 14.10]{NS06}, we have
\begin{align*}
\varphi^{p\mathcal{A}p}\left(x_{i_{\sigma(1)}}\cdots x_{i_{\sigma(n)}}\right) &= \sum_{\pi \in \mathrm{NC}(n)}\kappa_\pi^{p\mathcal{A}p}\left(x_{i_{\sigma(1)}}, \dots, x_{i_{\sigma(n)}}\right)\\
&= \varphi(p)^{-1}\sum_{\pi \in \mathrm{NC}(n)}\kappa_\pi\left(\varphi(p)a_{i_{\sigma(1)}}, \dots, \varphi(p)a_{i_{\sigma(n)}}\right)\\
&= \varphi(p)^{n-1}\varphi\left(a_{i_{\sigma(1)}}\cdots a_{i_{\sigma(n)}}\right)\\
&= \varphi(p)^{n-1}\varphi\left(a_{i_1}\cdots a_{i_n}\right)\\
&= \varphi(p)^{-1}\sum_{\pi \in \mathrm{NC}(n)}\kappa_\pi\left(\varphi(p)a_{i_1}, \dots, \varphi(p)a_{i_n}\right)\\
&= \sum_{\pi \in \mathrm{NC}(n)}\kappa_\pi^{p\mathcal{A}p}\left(x_{i_1}, \dots, x_{i_n}\right)\\
&= \varphi^{p\mathcal{A}p}\left(x_{i_1}\cdots x_{i_n}\right),
\end{align*}
where the second equality follows from \eqref{KappaAp}. Consider now the $2$-sequence $M = \{M_{m, n}\}_{(m, n) \in \mathbb{Z}_+^2}$ defined by $M_{m, n} = \varphi^{p\mathcal{A}p}(x_1^mx_2^n)$ for $(m, n) \in \mathbb{Z}_+^2 \backslash \{(0, 0)\}$ and $M_{0, 0}=1$. It is
easy to verify that the sesquilinear form $[\cdot,\cdot]$ induced by $M$ satisfies the conditions (1) and (2) in Theorem \ref{moment1}. Indeed, observe that for any $q = \sum_{j = 1}^\ell c_js^{m_j}t^{n_j} \in \mathbb{C}[s, t]$, the identity shown above yields
\begin{align*}
[q,q]_M&=\sum_{j,k=1}^\ell c_j\overline{c}_kM_{m_j+m_k,n_j+n_k} \\
&=\sum_{j,k=1}^\ell c_j\overline{c}_k\varphi^{p\mathcal{A}p}(x_1^{m_j+m_k}x_2^{n_j+n_k}) \\
&=\sum_{j,k=1}^\ell c_j\overline{c}_k\varphi^{p\mathcal{A}p}(x_1^{m_j}x_2^{n_j}x_1^{m_k}x_2^{n_k}) \\
&=\varphi^{p\mathcal{A}p}\left(\left(\sum_{j=1}^\ell c_jx_1^{m_j}x_2^{n_j}\right)\left(\sum_{j=1}^\ell c_jx_1^{m_j}x_2^{n_j}\right)^*\right) \geq 0.
\end{align*}
Similarly, one can see that $|[sq, q]_M| \leq L\cdot[q, q]_M$ and $|[q, tq]_M| \leq L\cdot[q, q]_M$, where $L = \max\{\|a_1\|, \|a_2\|\}$. Thus $M$ is a moment sequence whose representing measure $\mu$ is a compactly supported planar probability measure. By the moment-cumulant formula and the identity $\varphi^{p\mathcal{A}p}\left(x_{i_{\sigma(1)}}\cdots x_{i_{\sigma(n)}}\right)=\varphi^{p\mathcal{A}p}\left(x_{i_1}\cdots x_{i_n}\right)$ again, one can see that the desired identity in the lemma holds.
\end{proof}

\begin{thm}
Let $\mu$ be a compactly supported probability measure on $\mathbb{R}^2$, then there exists an additive convolution semigroup $(\mu_t)_{t \geq 1}$ of compactly supported probability measures on $\mathbb{R}^2$ such that
$$\mu_1 = \mu\;\;\;\;\;\text{and}\;\;\;\;\;\mu_s \boxplus\boxplus \mu_t = \mu_{s + t}$$
for all $s, t \geq 1$, and the mapping $t \mapsto \mu_t$ is continuous with respect to the weak$^*$ topology on planar
probability measures.
\end{thm}

\begin{proof}
Let $(a, b)$ be a two-faced pair in a $C^*$-probability space $(\mathcal{A}, \varphi)$ such that $a = a^*$, $b = b^*$, $[a, b] = 0$, and the distribution of $(a, b)$ is $\mu$. For any $t \geq 1$, let $p \in \mathcal{A}$ be a projection such that $\varphi(p) = 1/t$ and $p$ is free from $\{a, b\}$. Note that for given $\mu$, one can always realize such $(a, b)$ and $p$. For example, consider $a$ and $b$ to be the multiplication operators on the Hilbert space $L^2(\mathbb{R}^2, d\mu)$ with respect to the first and second coordinates of $\mathbb{R}^2$, respectively, and consider $p$ to be the multiplication operator by $\chi_{[0, 1/t]}$ on $L^2([0, 1], dx)$. Then $(a, b)$ and $p$ satisfy the required properties in the reduced free product of $\mathcal{L}^\infty(\mathbb{R}^2)$ and $\mathcal{L}^\infty([0, 1])$, the $C^*$-subalgebras of $B(L^2(\mathbb{R}^2, d\mu))$ and $B(L^2([0, 1], dx))$ consisting of multiplication operators induced by $L^\infty$ functions on $\mathbb{R}^2$ and $[0, 1]$, respectively.

Consider the two-faced pair $(p(ta)p, p(tb)p)$ in the compressed space $(p\mathcal{A}p, \varphi^{p\mathcal{A}p})$. By Lemma \ref{semigrouplem}, there exists a compactly supported probability measure $\mu_t$ on $\mathbb{R}^2$ such that
\begin{align*}
\kappa_{m, n}^{\mu_t} &= \kappa_{m + n}^{p\mathcal{A}p}(\underbrace{p(ta)p, \dots, p(ta)p}_{m\;\mathrm{times}}, \underbrace{p(tb)p, \dots, p(tb)p}_{n\;\mathrm{times}})\\
&= t\kappa_{m + n}(\underbrace{a, \dots, a}_{m\;\mathrm{times}}, \underbrace{b, \dots, b}_{n\;\mathrm{times}})\\
&= t\kappa_{m, n}(a, b) = t\kappa_{m, n}^\mu
\end{align*}
for all $m, n \geq 0$ with $m + n \geq 1$, where the second equality follows from \eqref{KappaAp}. This shows that
$$\kappa_{m, n}^{\mu_{s + t}} = (s + t)\kappa_{m, n}^\mu = \kappa_{m, n}^{\mu_s} + \kappa_{m, n}^{\mu_t}$$
for all $s, t \geq 1$, thus $\mu_{s + t} = \mu_s \boxplus\boxplus \mu_t$. Moreover, it is clear that
$\mu_1 = \mu$, and the mapping $t \mapsto t\kappa_{m, n}^\mu$ is continuous, hence all moments and cumulants of $\mu_t$ are continuous in $t$.
\end{proof}

\section*{Acknowledgements}

It is our pleasure to thank the referee for carefully reviewing the paper and providing many valuable corrections and suggestions.

\end{document}